\documentclass{amsart}
\usepackage{geometry}
\usepackage{setspace}

\usepackage{peambl}

\title{Homomorphic images of algebraic groups}
\author{Uri Bader and Elyasheev Leibtag }
\date{}

\begin{document}
\maketitle
\begin{abstract}
    We study topological group theoretic properties of algebraic groups over local fields. In particular, we find conditions under which such groups have closed images under arbitrary continuous homomorphisms into arbitrary topological groups.
\end{abstract}
\section{Introduction}

Let $k$ be a local field, and let $\bfG$ denote a linear algebraic group over $k$.
The group $\bfG(k)$ has a natural locally compact topology, called \emph{the $k$-point analytic topology}, making it a topological group.
In this paper, we are interested in properties of $\bfG(k)$ viewed as a topological group.

\begin{convention*}
Throughout this paper, unless explicitly specified otherwise, all topological groups are assumed to be Hausdorff, hence Tychonoff.
\end{convention*}

A topological group that admits no weaker (Hausdorff) group topology is called \emph{minimal}.
It is straightforward that every compact group is minimal.
Moreover, the center of a Polish minimal group must be compact, see Lemma~\ref{lem : minimal has compact center}.
In \cite{omori1966homomorphic} and \cite{goto1973absolutely}, H. Omori and M. Goto showed that for connected real algebraic groups this necessary condition is in fact sufficient - a connected real algebraic group with a compact center is minimal.
The following is an extension of the above result which applies to an arbitrary local field of characteristic zero, see Theorem~\ref{theorem : minimal condition algebraic groups char 0}.

\begin{mthm} \label{mthm : theorem A}
Let $k$ be a local field of characteristic zero, and let $\bfG$ be a connected $k$-algebraic group.
Then the group $\bfG(k)$ is minimal if and only if the center of $\bfG(k)$ is compact.
\end{mthm}

We note that the work of Omori and Goto uses the connectivity of the real group $G$ in an essential way.
Connectivity is a feature lacking over non-Archimedian local fields which are totally disconnected. As such, our proof of Theorem~\ref{mthm : theorem A} is entirely different from theirs.

It is a standard fact that 
a Polish topological group is minimal if and only if 
its image is closed under any \emph{injective} continuous group homomorphism into any topological group, see Corollary~\ref{cor : minimal has closed image} below.
A topological group is said to be \emph{sealed} if its image is closed under any (not necessarily injective) continuous group homomorphism into any topological group.
Clearly, a Polish sealed group must have a compact center for all of its quotients.
Relying on the work of Omori and Goto, 
Mayer showed in \cite{MR1428116} that 
for connected locally compact groups this necessary condition is in fact sufficient - a connected locally compact group such that all of its quotients have a compact center is sealed.
The following is a generalization that applies to connected algebraic groups over an arbitrary local field of characteristic zero, see Theorem~\ref{theorem : sealed theorem}.

\begin{mthm} \label{mthm : theorem B}
Let $k$ be a local field of characteristic zero and let $\bfG$ be a connected $k$ group with a $k$-Levi Decomposition, $\bfL\ltimes \Rad{u}{\bfG}$. The following are equivalent
\begin{enumerate}
    \item The group $\bfG(k)$ is sealed.
    \item For any normal unipotent $k$-algebraic subgroup $\bfN\lhd \bfG$ the center $Z((\bfG/\bfN) (k))$ is compact.
    \item The center of the reductive group $\bfL$ is anisotropic (equivalently, the center of $\bfL(k)$ is compact) and the action of $\bfL$ on $\Rad{u}{\bfG}$ is with no non-trivial fixed points, 
    \item The center of the reductive group $\bfL$ is anisotropic (equivalently, the center of $\bfL(k)$ is compact) and the action of $\bfL$ on $\Lie(\Rad{u}{\bfG})$ is with no invariant vectors.
\end{enumerate}
\end{mthm}

\subsection{Previous literature}
As already mentioned, minimality questions for connected groups were considered by
Omori, Goto and Mayer, see \cite{goto1973absolutely}, \cite{omori1966homomorphic} and \cite{MR1428116}.
In the special case of semisimple groups, Theorem~\ref{mthm : theorem A} and~\ref{mthm : theorem B}
were established by Bader-Gelander in \cite{MR3692904} over an arbitrary local filed.
In another direction, an interesting recent work by Megrelishvili and Shlossberg gives minimality criteria for groups of the form $\SL{n}{F}$, where $F\le k$ is an arbitrary subfield, see \cite{MR4495833}, as well as \cite{shlossberg2022minimality} concerning number theoretic applications.

For an extensive survey on minimality and closed image properties we refer the reader to \cite{MR3205486},\cite{MR1651174}, \cite{dikranjan1998categorically} and \cite{banakh2017categorically}.

Lately, Carter and Willis investigated the sealed property in the context of locally compact totally disconnected groups arising as automorphism groups of buildings, see \cite{carter2022homomorphic}.

Regarding infinite-dimensional isometry groups, Duchesne showed that the Polish topology of the isometry group of the infinite-dimensional hyperbolic space is minimal
 \cite{duchesne2020polish}.

In a recent work by Ghadernezhad and de la Nuez González, the authors investigate minimality of the automorphism group of homogeneous structures, see \cite{ghadernezhad2019group}. Moreover, in \cite{de2022compact}, de la Nuez González points out that the compact-open topology on the
homeomorphism group of a surface without
boundary is minimal.

\subsection{A note on positive characteristic}
In proving Theorems~\ref{mthm : theorem A} and~\ref{mthm : theorem B}, we use the existence of a Levi Decomposition over $k$, and that the connected $k$-unipotent radical is $k$-split. Moreover, in proving Theorem~\ref{mthm : theorem B}, we use the fact that the Levi factor is completely reducible. These properties are always  satisfied in zero characteristic, but do not generally hold over fields of positive characteristics, hence the zero characteristic assumption we take in Theorems~\ref{mthm : theorem A} and~\ref{mthm : theorem B}.
However, it seems that even over a field of positive characteristic, these assumptions could be removed in many cases, maybe in all cases.
This direction of study will be carried elsewhere.
Yet, as a ``proof of concept", we show in Section \ref{Positive char} how the tools developed in this paper may be used to reprove the aforementioned result of Bader-Gelander, namely proving Theorems~\ref{mthm : theorem A} and~\ref{mthm : theorem B} 
in the special case of semisimple groups over an arbitrary local field.

\subsection{Description of the paper}
In section \ref{Baby case}
we introduce a baby case presenting the main techniques and ideas that will be used throughout the paper. In particular, the impotent role played by admitting a continuous linear action on a finite-dimensional topological vector space.
In section \ref{Preliminaries} we recall a few topological properties of $k$-points of algebraic groups. We also provide auxiliary lemmas concerning groups acting by automorphisms on locally compact abelian groups.
Section \ref{Minimal and Sealed groups} is devoted to a brief survey of the minimal and sealed group properties, as well as the relative properties for subgroups. The importance of these relative properties is that in order to show minimality, we show relative minimality and co-minimality of a subgroup, usually the normal subgroup under a semi-direct product structure.

We point out that the standard topology on a finite-dimensional vector space over a local field \emph{is not} minimal. 
In section \ref{Rigidity of scalar multiplication} we show that the standard topology on a finite-dimensional vector space over a local field \emph{is} minimal amongst all group topologies admitting a continuous action by an adequate multiplicative subgroup.
Section \ref{Linear algebraic action of Tori} shows that if a vector space is a representation space of an algebraic torus then the torus acts on the vector space by an adequate multiplicative subgroup of the field, thus making the vector space minimal with respect to a continuous torus action.

In section \ref{Minimal algebraic groups} we prove Theorem~\ref{mthm : theorem A}.
We first prove it for solvable groups, using the linear action of the torus on the unipotent radical, and then treat the general case by finding a co-compact solvable subgroup (the solvable radical of a minimal parabolic) and showing that it is minimal.
Section \ref{Sealed algebraic groups} is devoted to the proof of Theorem~\ref{mthm : theorem B}. As in the baby case of $\PGL{2}{k}$, we first show how minimality of simple groups implies the sealed property for these groups. For general algebraic groups, we show by induction on the nilpotency degree in the unipotent radical that the radical is relatively sealed.
In section \ref{Positive char} we show how the tools developed in this paper can be used to prove previously known results for semisimple groups over local fields of any characteristic.
\subsection{Acknowledgments}
We are grateful to Gil Goffer, Tal Cohen, and Yotam Hendel for many remarks and suggestions for improvements.
We would like to also thank Dikran Dikranjan and Michael Megrelishvili for their informative comments and remarks.

Above all, we thank the members of the Midrasha on Groups at the Weizmann Institute for their support, friendship, and professional encouragement.
\section{Baby case}
\label{Baby case}
In this section, we clarify some of the ideas involved in our proofs by applying them to the special
cases  (already well-know)$A=k^*\ltimes k^n$ and $\bfG=\mathrm{PGL}_2$,
see Propositions~\ref{prop : baby case minimal}, \ref{prop : second baby case minimal} and \ref{prop : baby case sealed}.
In the subsequent sections, we will use the general results presented in this section.
Here $k$ is an arbitrary local field.
\subsection{Some general facts}

\begin{thm}[Banach, \cite{banach1931metrische}] 
\label{thm : banach open mapping theorem}
    Every continuous surjective morphism between Polish groups is open.
\end{thm}
\begin{cor}
\label{corollary : bijection onto polish group makes minimal}
    Let $G$ be a Polish group with topology $\s$. If $G$ admits a weaker group topology $\w$, and a bijective $\w$-continuous homomorphism $f\colon G\to H$ onto a Polish group $H$, then $\w=\s$.
\end{cor}
\begin{proof}
    Since $\w$ is weaker than $\s$, the map $f$ is $\s$-continuous as well. By the Open Mapping Theorem~\ref{thm : banach open mapping theorem}, the map $f$ is a homeomorphism onto $H$. For any open set $W\in \s$, $f(W)$ is open in $H$, since $f$ is $\w$-continuous, $f^{-1}(f(W))=W$ is open in $\w$, thus $\w=\s$.
\end{proof}
\begin{lemma}[Merson, Lemma 7.2.3 \cite{MR1015288}] 
\label{lemma : Mersons lemma}
    Let $G$ be a topological group with topology $\s$, and let $\w\le \s$ a weaker group topology. If there exists a subgroup $H\le G$, such that both the subspace topologies coincide on $H$, and both quotient topologies coincide on $G/H$, then $\w=\s$.
\end{lemma}
\begin{proof}
    Let $S\in \s$ be a $\s$-open identity neighborhood. We will show that there exist $W\in \w$, a $\w$-open identity neighborhood, with $W\subset S$.
    Let $S^{\frac{1}{2}}\in \s$ be a symmetric identity neighborhood such that $S^{\frac{1}{2}}S^{\frac{1}{2}}\subset S$. 
    Since both subspace topologies coincide on $H$, there exists $U\in \w$ an open identity neighborhood, with $U\cap H= S^{\frac{1}{2}}\cap H$.
    Let $U^{\frac{1}{2}}\in \w$ be a symmetric open identity neighborhood such that $U^{\frac{1}{2}}U^{\frac{1}{2}}\subset U$. 
    Since both quotient topologies coincide, we have that the $\s$-open subset $(S^{\frac{1}{2}}\cap U^{\frac{1}{2}})H$ is $\w$-open as well.
    We set $W= (S^{\frac{1}{2}}\cap U^{\frac{1}{2}})H\cap U^{\frac{1}{2}}\in \w$.
        Fixing $y\in W$ we find $x\in S^{\frac{1}{2}}\cap U^{\frac{1}{2}}$ and $h\in H$ such that $y=xh$.
    Then $h=x^{-1}y\in U^{\frac{1}{2}}U^{\frac{1}{2}}\subset U$. So $h\in U\cap H\subset S^{\frac{1}{2}}$, making $y=x h \in S^{\frac{1}{2}}S^{\frac{1}{2}}\subset S$, as needed.
\end{proof}
A general strategy for proving minimality of a group $G$, using the Merson Lemma, is by showing that there exists a subgroup $H\le G$, such that for all weaker group topologies, the subspace topology on $H$ and the quotient topology on $G/H$ coincide with the a-piori stronger subspace and quotient topologies.
\subsection{Weaker additive group topologies on vector spaces}

Let $k$ be a local field with absolute value $| \cdot|$, and let $V$ be a $n$-dimensional $k$-vector space with norm $\| \cdot \|$. 
We consider the topology on $k$ to be the locally compact field topology, and the topology on $V$ to be the locally compact norm topology.

\begin{lemma}
	\label{lemma :  weaker topology on V unbounded open sets}
	Any additive group topology on $V$ which is strictly weaker than the norm topology does not contain bounded open sets.
\end{lemma}
\begin{proof}
    Let $\w$ be a weaker topology on $V$
	and assume that there exists a bounded open $0_V$-neighborhood $W\in \w$ with $\|W\|<\lambda$.
	We will show that every norm-open ball contains an $\w$-neighborhood, and hence that the norm topology is equivalent to $\w$.
	Let $B_\e\subset V$ be the norm-open ball of radius $\e>0$.
	Denote by $\overline{B_\lambda}$ the norm-closed ball of radius $\lambda$, and
	$\ann(\e,\lambda)\coloneqq\overline{B}_\lambda\setminus B_\e$.
	The annulus $\ann(\e,\lambda)$ is compact in the norm topology, so it is compact with respect to a weaker topology $\w$ as well. 
	As any group topology is Tychonoff, there exists a $\w$-open $0_V$-neighborhood $W'$, such that $W'\cap  \ann(\e,\lambda)=\emptyset$.
	Finally, we get that the $\w$-open set $W'\cap W$ is contained in the open ball $B_\e$, as needed.
\end{proof}
\begin{lemma}
	\label{lemma : baby case multiplication by ball}
	Let $B_\lambda$ be the open ball of radius $\lambda$ in $k$, and let $m\colon B_\lambda \times V \to V$ be the scalar multiplication map. Then there does not exist a strictly weaker additive group topology on $V$ such that the multiplication map $m$, is a continuous map.
\end{lemma}
\begin{proof}
	As all norms on finite dimensional vector spaces induce the same topology, we may assume without loss of generality that $\|\cdot\|$ is an $\mathrm{L}^\infty$ norm.
	Assume by contradiction that $\w$ is a strictly weaker group topology on $V$ such that the scalar multiplication map
	$m\colon B_\lambda \times V \to V$ is continuous with respect to $\w$.
	The unit sphere $S_1$ is norm-compact and therefore $\w$-compact.
	Let $W\in \w$ be a $0_V$-neighborhood not intersecting the norm sphere $S_1$. 
	By continuity of $m $ there exist $W'\in \w$ a $0_V$-neighborhood, and $\e<\lambda$ such that $m(B_\e \times W')\subset W$.
	By assumption, $\w$ is strictly weaker than the norm topology, hence by Lemma~\ref{lemma :  weaker topology on V unbounded open sets} $W'$ is unbounded. 
	Let $w'\in W'$ be a vector such that $\e^{-1}<\|w'\|$. Since the norm is $\mathrm{L}^\infty$ there exists a scalar $x\in k$ such that $|x|=\|w'\|^{-1}$. We get that $(x,w')\in B_\e\times W'$, making $ m(x,w')=x w'\in S_1$ as $\|x w'\|=|x|\cdot \|w'\|=1$, in contradiction to $W$ not intersecting the unit sphere.
\end{proof}
In particular, we get as a corollary a fundamental theorem in functional analysis. 
\begin{cor}[{\cite[Chapter I.\S 2.3 Theorem 2]{bourbaki1987topological}}] 

    Let $k$ be a local field. Any $n$-dimensional topological $k$-vector space $V$ is homeomorphic to the topological norm space $k^n$. 
\end{cor}
\begin{proof}
    Let $V$ be an $n$-dimensional topological vector space over a local field $k$. This means that $V$ admits an additive group topology $\tau$, such that the scalar multiplication map
    $m\colon k \times V \to V$ is continuous.
    Let$\{e_i\}_{i=1}^n$ be a basis of $V$.
    For any $i=1,...,n$ the map
    $$m_i\colon k\to V \ \ m_i\colon k\mapsto k\cdot e_i$$
    is continuous. And since $\tau$ is an additive group topology we get a continuous additive group homomorphism from the product $k^n$ onto $V$.
    $$\Psi\colon k^n \to V \ \ \Psi \colon (k_1,...,k_n)\mapsto \sum_{i=1}^n k_i\cdot e_i$$
    The continuity of $\Psi$ implies that the $\tau$ topology on $V$ is weaker than the norm topology. 
    By Lemma~\ref{lemma : baby case multiplication by ball} above, $\tau$ is equal to the norm topology, and therefore the topology is unique.
\end{proof}
\subsection{Minimality of the homothety group}
Equipped with the product topology, the continuous action of $k^*$ on $V$ makes $A=k^*\ltimes V$ a topological group. We call the group $A$ the homothety group of $V$, and we call the standard product topology on $A$ the analytic topology. 
\begin{lemma}
\label{lemma: baby case reletive minimality}
	Let $\w$ be a group topology on $A$ weaker than the analytic topology. Then the restriction of $\w$ to the subspace $V$ is equal to the analytic norm topology.
\end{lemma}
\begin{proof}
	Let $\alpha$ be the action map of $k^*$ on $V$, that is,
	$$ \alpha \colon k^*\times V \to V\times V, \ \ \alpha\colon (r,v)\mapsto (rv,v)$$
    Since $\w$ is a group topology on $A$, the action by conjugation is continuous with respect to the $\w$-subspace topologies, thus the map $\alpha$ is continuous with respect to the $\w$-subspace topologies on $k^*$ and on $V$. 
    By strengthening the topology on the domain, the map $\alpha$ is still continuous. 
    So by letting $V$ admit the $\w$-subspace topology $\w_V$, and embellishing $k^*$ with analytic topology,
    the map $\alpha$ remains continuous.
	Since $\w_V$ is an additive group topology on $V$, the subtraction operation 
	$$\mathrm{-} \colon V\times V \to V, \ \ \mathrm{-} \colon (v,w)\mapsto v-w$$
	is $\w_V$-continuous.
	Observe that on the local field $k$, the addition by $1$ map, $\mathrm{+_1}\colon B_1\to k^*$ mapping $a$ to $a+1$ is continuous.
	We get that the map $m\coloneqq \mathrm{-}  \circ \alpha\circ (\mathrm{+_1}\times Id)$
	$$m\colon B_1 \times V \to V \ \ m\colon (x,v)\mapsto (x+1)v-v=xv $$
    is a continuous multiplication by scalars in the unit ball.
	By Lemma~\ref{lemma : baby case multiplication by ball} above, the restricted topology $\w_V$ on $V$ is equal to the norm topology.
\end{proof}
\begin{lemma}
	\label{lemma: baby case co minimality}
	Let $\w$ be a group topology on $A$, which is weaker than the analytic topology. Then the $\w$-quotient topology on $A/V$, $\w_{A/V}$, is equal to the analytic quotient topology on $A/V$, which is isomorphic to $k^*$.
\end{lemma}
\begin{proof}
	The subgroup $V$ is normal in $A$,
	so the action by conjugation $A\times V\to V$ is continuous with respect to the $\w$-subspace topologies.
	Since $V$ is in the kernel of the action, we get continuous action of the quotient group $A/V$ on $V$.
	As abstract groups $A/V$ is isomorphic to $k^*$ and the action on $V$ is by scalar multiplication on the vector space $V$. 
	By Lemma~\ref{lemma: baby case reletive minimality}, the $\w$-subspace topology $\w_V$, is equal to the norm topology on $V$. 
    We, therefore, get a continuous representation $\pi \colon A/V\to \GL(V)$ onto the closed analytic subgroup of scalars, $k^*\le \GL(V)$. 
	Thus $\pi$ is a continuous bijection onto a Polish group, so by Corollary \ref{corollary : bijection onto polish group makes minimal}, the $\w_{A/V}$-quotient topology is equal to the Polish analytic topology on $A/V$.
\end{proof}
\begin{prop}[see {\cite[Theorem 4.7(b)]{MR1651174}}]
\label{prop : baby case minimal}
    The homothety group $k^*\ltimes V$, endowed with the analytic topology, is minimal.
\end{prop}
\begin{proof}
    Let $\w$ be a group topology on $A=k^*\ltimes V$, weaker than the analytic topology.
    By Lemma~\ref{lemma: baby case reletive minimality}, the subspace topology $\w_V$ is equal to the analytic subspace topology, and by Lemma~\ref{lemma: baby case co minimality} the $\w$-quotient topology on $A/V$ is equal to the quotient of the analytic topology. Hence by Lemma~\ref{lemma : Mersons lemma}, $\w$ is equal to the analytic topology.
\end{proof}
\subsection{Minimality and sealedness of \texorpdfstring{$\PGL{2}{k}$}{G} }
\begin{prop}[{\cite[Theorem 5.4]{MR3692904}}]
\label{prop : second baby case minimal}
    The group $\PGL{2}{k}$, endowed with the analytic topology, is minimal.
\end{prop}
\begin{proof}
    Let $\w$ be a group topology on $G=\PGL{2}{k}$, weaker than the analytic topology.
	The Borel subgroup $P$, is isomorphic to the homothety group $k^*\ltimes k$, hence by Proposition~\ref{prop : baby case minimal}, $P$ is minimal. Thus for any weaker group topology $\w$, the subspace topology on $P$ coincides with the analytic topology.
	Moreover, since the analytic topology is Polish, $P$ is $\w$-closed. Hence the $\w$-quotient topology on $G/P$ is Hausdorff and since the analytic quotient topology on $G/P$ is compact it does not admit any strictly weaker typology.
	By Merson's Lemma~\ref{lemma : Mersons lemma} $\w$ is equal to the analytic topology on $G$.
\end{proof}
\begin{prop}[{\cite[Theorem 5.1]{MR3692904}}]
\label{prop : baby case sealed}	
    The group $\PGL{2}{k}$, endowed with the analytic topology, is sealed.
\end{prop}
\begin{proof}
	Let $f\colon G=\PGL{2}{k}\to H$ be a continuous homomorphism into a topological group $H$.
	By \cite[Chapter I, Theorem (1.5.6)(i) and Theorem (2.3.1)(b)]{margulis1991discrete}, any non-trivial normal subgroup of $G$ is co-compact, thus if $f$ is not injective $f(G)$ is compact and obviously closed.
	We thus assume as we may that $f$ is injective.
	By Proposition~\ref{prop : second baby case minimal}, the $f$-initial topology $\tau^f$ must coincide with the analytic topology on $G$.
    Therefore, we get that $f$ is a homeomorphism onto its image.
    Therefore the subspace topology on $f(G)$ is Polish, hence closed, see \cite[Theorem 2.16]{melleray2016polish}.
\end{proof}

\section{Preliminaries}
\label{Preliminaries}
\subsection{Algebraic groups over local fields}
We assume that the reader is familiar with the basic theory of algebraic groups.
 
Throughout this section we let $k$ be a local field (of any characteristic), with absolute value $|\cdot|$, and denote by $\bar{k}$ its algebraic closure. 
Algebraic groups over $k$ will be identified with their $\bar{k}$ points and will be denoted by boldface letters. Their $k$-points will be denoted by corresponding Roman letters.
We regard $G=\bfG(k)$ as a topological group equipped with a $k$-analytic structure. 

As seen in the next two propositions, for a local field $k$, the geometric properties of the group $\bfG$ interact nicely with the topological $k$-point structure on $G$. Namely, under a $k$-morphism of algebraic groups the image of the $k$-points is topologically closed, and geometrically (Zariski) dense.
\begin{prop}[{\cite[Chapter I, Proposition (2.5.3)(ii)]{margulis1991discrete}}]
    \label{prop : zariski dnese of k}
    Let $\bfG$ be a connected $k$ group, then $\bfG(k)$ is Zariski dense in $\bfG$.
\end{prop}
\begin{prop}
    \label{prop : image and kernel of k point maps}
    Let $\bfG$, $\bfH$ be $k$-algebraic groups, and $f\colon \bfG \to \bfH$ be a $k$-morphism. Then the image of the $k$-points, $f(G)$, is a closed subgroup of $H$.
    Moreover, the kernel of $f$, $\ker(f)$ is a $k$-algebraic group such that as $k$-analytic topological spaces $\bfG(k)/ {\ker(f)}(k)\cong f( G)$.
\end{prop}
\begin{proof}
    This follows by \cite[Proposition 2.2]{MR3668057}, identifying $\bfH$ as a $k$-variety with a $\bfG$ action given by left multiplication via $f$, i.e. $g.h=f(g)h$.
\end{proof}
\begin{rem}
\label{remark : k points of image}
    Let $\bfG$, $\bfH$ and $f$ be as above.
    Note that $f(G)$ is Zariski dense in the $k$-algebraic group $f(\bfG)$. But generally, $f(G)$ is not onto the $k$-points $f(\bfG)(k)$. For example let $k=\bbR$, take $\bfG=\mathrm{SL}_{2}$ and let $f$ be the central quotient map. Observe that
    $f(G)=\PSL{2}{\bbR}\subsetneq\PGL{2}{\bbR}=f(\bfG)(k)$
\end{rem}
As seen in the remark above, one should be careful when dealing with $k$-points of quotient varieties, as these may not be $k$-points of a $k$-algebraic quotient.
We recall that in the case of a parabolic subgroup, or in the case of split solvable groups, the quotient of the $k$-points is equal to the $k$-points of the quotient group.
\begin{definition}[{\cite[Definition 15.1]{borel2012linear}}]
    A connected solvable $k$-group is called \textit{split} if it admits a composition series by $k$-subgroups with successive quotients $k$-isomorphic to $\bfG_m$ or $\bfG_a$.
\end{definition}
\begin{rem}[{\cite[Theorem 15.4]{borel2012linear}}]
    If the field $k$ is perfect, then any connected unipotent group is $k$-split.
\end{rem}
\begin{prop}[{\cite[Corollary 15.7]{borel2012linear}}]
\label{prop : qutient split solvable}
    Let $\bfG$ be a connected $k$-algebraic group and $\bfA$ a normal connected $k$-split solvable subgroup. Then the $k$-point algebraic map $\pi\colon G \to \bfG/\bfA (k)$ is surjective, i.e. $G/A=\bfG/\bfA (k)$.
\end{prop}
\begin{definition}
    Let $\bfG$ be a connected $k$-algebraic group. A subgroup $P\le G$ is called parabolic if 
    $P=\bfP(k)$ for some $k$-parabolic subgroup $\bfP<\bfG$.
\end{definition}
\begin{lemma}
\label{lem : center of parabolic}
    Let $\bfG$ be a connected $k$-algebraic group, and $P\le G$ a parabolic subgroup. Then the center of the parabolic subgroup, $Z(P)$, equals the center of the group, $Z(G)$.
\end{lemma}
\begin{proof}
    Denote by $\bfP$ the Zariski closure of $P$. This is a parabolic subgroup.
    Since $G$ is Zariski dense in $\bfG$, $Z(G)$ is central in $\bfG$, hence contained in $\bfP$.
    We conclude that $Z(G)<Z(P)$. 
    We will argue to show the other inclusion, $Z(P)<Z(G)$.
    
    For any element in the center of $P$, $z\in Z(P)$, we may define a $k$-morphism of algebraic varieties 
    \[
    \mathbf{c}_z\colon \bfG \to \bfG  \ \ g\mapsto [g,z]
    \]
    The map $\mathbf{c}_z$ is a $k$-morphism of algebraic varieties, as it is defined by group operations in $\bfG$. 
    Since $P$ is Zariski dense in $\bfP$, and $P\subset \mathbf{c}_z^{-1}(\{e\})$, we get that $\bfP\subset \mathbf{c}_z^{-1}(\{e\}) $. Implying that for any $g\in \bfG$, $p\in \bfP$, $\mathbf{c}_z(gp)=\mathbf{c}_z(g)$.
    So the map $\mathbf{c}_z$ factors to a $k$-morphism $\hat{\mathbf{c}}_z\colon \bfG/\bfP \to \bfG$.
    Since $\bfP$ is parabolic, the quotient variety $\bfG/\bfP$ is a connected complete $k$-variety. Therefore the map $\hat{\mathbf{c}}_z$ must be constant.
    Thus $z\in Z(\bfG)$, and in particular in $Z(G)$, as needed.
\end{proof}

\subsubsection{ Algebraic linear action} \label{subsec:Ala}
Assume $k$ is of characteristic zero, and $\bfU$ is a connected unipotent $k$-group.
We have a $k$-isomorphism of algebraic varieties $\bfU\cong \Lie(\bfU)$. 
Thus letting $\bfG$ be a $k$-group acting $k$-algebraically by group automorphisms on a connected abelian $k$-unipotent group, we tacitly assume that the $G$ action on $U$ is linear, where the linear structure is induced by the $k$-isomorphism with the Lie algebra.

In the case that $k$ is of positive characteristic, auxiliary assumptions on the $k$-groups $\bfG$ and $\bfU$ are needed, see \cite{mcninch2014linearity} for a detailed discussion.

\subsection{Faithful actions on vector spaces}
This subsection is devoted to proving proposition~\ref{prop : faithfull irreducable action normal subgroup contined} bellow.
\begin{lemma}
\label{lem : normal subgroup in faithful action}
Let $L$ be a group acting faithfully by automorphisms on an abelian group $A$, and denote $G=L\ltimes A$. Then any non-trivial normal subgroup $M\lhd G$ intersects $A$ non-trivially.
\end{lemma}
\begin{proof}
Assume by contradiction that $M\cap A=\{e\}$. 
Observe that $[M,A]\subset M\cap A= \{e\}$, let $m=l_m a_m$ be an element in $M$. For any $a\in A$, $[a,m]=e$, so $ a=mam^{-1}=l_m a l_m^{-1} $. By the faithfulness of the $L$-action, we get that $l=e$. Therefore $M\subset  M\cap A=\{e\}$ is trivial in contradiction to $M$ being non-trivial. 
\end{proof}
\begin{lemma}\label{lem : faithful action on qutient}
Let $k$ be a local field, and $V$ be a finite-dimensional $k$-representation of a group $L$. 
Let $N\le V$ be an $L$ invariant additive subgroup. 
If the $L$ representation $V$ is faithful and irreducible, then the $L$ action on the quotient $ V/N$ is faithful.
\end{lemma}
\begin{proof}
By self-duality of local fields,  the Pontryagin dual of the vector space $V$, $\widehat{V}$, is a vector space isomorphic to $V$.
The $L$ action on the Pontryagin dual  $\widehat{V}$ is defined by pre-composition and is linear irreducible and faithful as well.
Let $N^{0}=\left\{ \chi\in\widehat{V}:\ \forall n\in N\ \chi(n)=1\right\} $ be the annihilator of the subgroup $N$. The annihilator $N^0$ is $L$ invariant, so by irreducibility $N^0$ contains a linear basis of the dual space $\widehat{V}$.
Therefore the $L$ action on $N^{0}$ is faithful making the $L$ action on the dual $\widehat{N^0}\cong V/N$ faithful as well.
\end{proof}
\begin{prop}
\label{prop : faithfull irreducable action normal subgroup contined}
Let $k$ be a local field, and $V$ be a finite-dimensional $k$-representation of a group $L$ and set  $G=L\ltimes V$.
If the $L$ representation $V$ is faithful and irreducible, then any normal subgroup $M\lhd G$ is either contained in $V$ or contains $V$.
\end{prop}
\begin{proof}
Assume $V\nsubseteq M$, and denote $N = M\cap V$.
The group $L$ acts by automorphisms on $V/N$ and we may identify the quotient $G/N$ with $ L\ltimes V/N$. By lemma~\ref{lem : faithful action on qutient} the action of $L$ on $V/N$ is faithful. The normal subgroup $M/N\lhd G/N$ intersects $V/N$ trivially, hence by lemma~\ref{lem : normal subgroup in faithful action} $M/N$ is trivial implying that $N=M$, thus $M\subseteq V$ as needed.
\end{proof}
\begin{prop}
\label{prop : not containing subspace compact*}
    Let $k=\bbQ_p$ for some prime $p$, and let $V$ be an $n$-dimensional $k$-vector space. 
    Any additive subgroup $N\le V$ which does not contain a vector subspace is compact.
\end{prop}
\begin{proof}
    We will show that $N$ is compact by showing its Pontryagin dual is discrete.
    Observe that if $N^0$ is contained in a vector subspace $W\subset \widehat{V}$, then $N$ contains the linear subspace $W^0\coloneqq \{ v\in V : \forall\chi\in W \ \chi(v)=0 \}$.
    Since $N$ does not contain a linear subspace we get that $\spn_{\bbQ_p}(N^0)= \widehat{V}$.
    Hence, $N^0$ contains a basis, and since $N^0$ is $\bbZ$-invariant and closed it is $\bbZ_p$-invariant and is, therefore, an open subset in $\widehat{V}$.
    Therefore, $\widehat{V}/N^0$ is discrete and since
    $\widehat{N}\cong \widehat{V}/N^0$, we get that $N$ is compact.
\end{proof}

\section{Minimal and Sealed groups}
\label{Minimal and Sealed groups}

The purpose of this section is to survey two topological properties of groups, called minimality and sealdness, as well as a relative version of these properties.
The objects studied in this article are $k$-points of algebraic groups over a local field $k$.
Therefore we will pay attention to the case where $G$ is a locally compact Polish topological group. 
In order to keep this article self-contained, we include proofs for the special case of locally compact Polish groups. However, many of the lemmas are true for a more general class of topological groups.
For a full and comprehensive account of minimal group topologies (as well as the relative notions for subgroups) and closed image properties we refer the ready to  \cite{MR2556080}, \cite{MR3205486}, \cite{MR1651174}, \cite{dikranjan1998categorically} and \cite{banakh2017categorically}.

\subsection{Minimal topologies on groups}
Let $G$ be a group, with a group topology $\tau$. That is $\tau$ is a (Hausdorff) topology on $G$ such that the multiplication and inverse operations are continuous. Any subgroup of a topological group is a topological group with the subspace topology. We will denote the subspace topology on $H\le G$ by $\tau_H$, and denote the quotient topology on $G/H$ by $\tau_{G/H}$.
Given two group topologies $\w$ and $\tau$ on $G$, we say that a group topology $\w$ is weaker than $\tau$ if $\w\le \tau$, that is, any $\w$-open set is $\s$-open. We say $\w$ is strictly weaker than $\s$, if $\w\le \s$ and $\w\ne \s$.
\begin{definition}
    A topological group $G$ is called \textit{minimal} if it does not admit a strictly weaker (Hausdorff) group topology.
\end{definition}
\begin{example}
    Compact groups are minimal. Observe that the identity map between the strong and weak topologies is a continuous and closed bijection.
\end{example}
\begin{example}
\label{example : popular example}
    The homothety group $\bbR^*\ltimes \bbR$ is minimal (See proposition~\ref{prop : baby case minimal}).
\end{example}
Example \ref{example : popular example}, is a well-known result (see \cite{dierolf1979examples}), the minimality of the homothety group follows particularly from the work of Omori \cite[Theorem 1.1]{omori1966homomorphic}, which classifies all minimal real algebraic groups.
%
\begin{example}[Omori]
\label{example : omori}
    Any real algebraic group with a compact center is minimal.
\end{example}
\begin{definition}
\label{defn : initial top}
     Let $G$ and $H$ be topological groups, and $f\colon G\to H$ be an injective continuous homomorphism. The topology on $G$ generated by inverse images of open subsets in $H$ is called the \textit{$f$-initial topology} and is denoted $\tau^f$. 
\end{definition}
\begin{rem} \label{remark : minimal group is homeo to image} 
    The initial topology $\tau^f$ is a group topology weaker than the given topology on $G$, and the map $f$ is open with respect to the initial topology. 
    Therefore, if $G$ is minimal then any injective continuous homomorphism is a topological embedding onto the image.
\end{rem}
\begin{thm}[Stoyanov, Prodanov\cite{stoyanov1984every}]
\label{thm : Polish/LCA not minimal}
    Let $G$ be a Polish abelian group. Then $G$ is minimal if and only if it is compact.
\end{thm}
\begin{proof}
    By the Stoyanov and Prodanov Theorem \cite[Theorem 2.7.7]{MR1015288}, any abelian minimal group is pre-compact. Hence Polish abelian minimal groups are compact.
    In this paper, we are interested in locally compact Polish groups, for sake of completeness, we provide a short argument in the case $G$ is locally compact abelian.
    Let $\beta\colon G\to \beta G$ be the Bohr compactification of the locally compact abelian group $G$, see \cite[VII.5]{katznelson2004introduction}.
    Since minimal groups are homeomorphic to their images, we get that $G$ is homeomorphic to $\beta(G)$, and since $G$ is locally compact, $\beta(G)$ is locally compact and hence closed, implying that $G$ is compact.
\end{proof}
\begin{lemma}
\label{lemma : weaker than second countable}
    Let $G$ be a second countable topological group.
    For any $\w$, a weaker group topology on $G$ there exists a group topology $\w'\le \w$, such that $\w'$ is second countable.
\end{lemma}
\begin{proof}
    Let $\s$ denote the given second countable topology on $\s$, and let $\w\le \s$ be a weaker group topology.
    We will show that there exists a second countable weaker group topology $\w'\le\w$.
    
    Let $\{W_\alpha \}_{\alpha\in A}\in \w$ be a system of open identity neighborhoods.
    Since the group topology $\w$ is Tychonoff, $\cap_{\alpha\in A} \overline{W_\alpha} =\cap_{\alpha\in A} W_\alpha=\{e\} $, where $ \overline{W_\alpha}$ is the $\w$-closure.
    Observe that $\overline{W_\alpha}$ is $\s$-closed as well. Thus we obtain an open cover $\bigcup_{\alpha \in A}(G\setminus W_\alpha) = G\setminus \{e\} $.
    The $\s$-subspace topology on $G\setminus\{e\}$ is second countable and hence Lindelof, implying that there exists a countable subcover $\{W_i\}_{i=1}^\infty$ such that $\bigcup_{i=1}^\infty(G\setminus W_i) = G\setminus \{e\}  $, particularly $\cap_{i=1}^\infty {W_i}=\{e\}$.
    By \cite[Proposition 5.6]{MR3692904} there exist a weaker, a priori not necessarily Hausdorff, group topology $\w'\le \w$, such that $\w'$ is second countable and contains the collection $\{W_i\}_{i=1}^\infty$.
    But since $\cap_{i=1}^\infty {W_i}=\{e\} $, the topology $\w'$ is indeed Hausdorff and therefore $\w'$ is a second countable weaker group topology.
\end{proof}
\begin{cor}
\label{corollary : enough second countable}
    Let $G$ be a second countable group. If $G$ does not admit any weaker second-countable group topologies then $G$ is minimal.
\end{cor}
\begin{proof}
    Let $\s$ denote the given second countable topology on $G$, and let $\w\le \s$ be a weaker group topology. By Lemma~\ref{lemma : weaker than second countable} there exist a weaker group topology $\w'\le \w$, such that $\w'$ is second countable, therefore $\s=\w'$ implying that $\s=\w=\w'$, as needed.
\end{proof}
\subsection{Relative minimality}
The reader may benefit from keeping in mind the following example.
\begin{example}
Consider the group $\bbR^* \ltimes \bbR$ with the standard topology, and the subgroup $\bbR\le \bbR^* \ltimes \bbR $. Although the group $\bbR$ is not minimal, it does have the following property (see Lemma~\ref{lemma: baby case reletive minimality}): any weaker topology on $\bbR^* \ltimes \bbR$ agrees with the standard topology on the subgroup $\bbR$.
\end{example}
\begin{definition}
    Let $G$ be a topological group, and let $U\le G$ be a closed subgroup.
    We say that \textit{$U$ is minimal in $G$} if all weaker group topologies on $G$ coincide on the subgroup $U$.
\end{definition}
In the context of Polish topological groups, minimality can be interpreted in terms of a closed image property.
\begin{lemma}
\label{lem : minimal implies closed injective}
    Let $G$ be a Polish topological group, and let $U\le G$ be a closed subgroup. Then $U$ is minimal in $G$ if and only if for any injective continuous homomorphism $f\colon G\to H$ into a topological group $H$, the image $f(U)$ is closed in $H$.
\end{lemma}
\begin{proof}
    Assume that $U$ is minimal in $G$.
    Let $H$ be a topological group and $f \colon G \to H$ be an injective continuous homomorphism with a dense image.
    The $f$-initial topology $\tau^f$ is a weaker group topology on $G$, which must coincide on $U$. 
    Therefore, we get that the restriction $f_U\colon U\to H$ is an open bijection onto $f(U)$.
    Therefore the subspace topology on $f(U)$ is Polish hence $f(U)$ is closed, see \cite[Theorem 2.16]{melleray2016polish}.
  
    Now assume that the image of $U$ is closed under any injective map into a topological group. 
    Denote by $\s$ the given topology on $G$, and let $\w\le \s$ be a weaker topology on $G$.
    Following Corollary~\ref{corollary : enough second countable} we assume without loss of generality that $\w$ is second countable.

    Let $\iota\colon G \to \overline{G}$ be the completion of $G$ with respect to the two-sided uniformity defined by the group topology $\w$, see \cite[Chapter III, \S 3, no. 4 ]{bourbaki2013general} and\cite{raikov1946completion}, this completion is also known as the Raikov-completion of the group. 
    
    Since $\w$ is second countable, it follows by \cite[Chapter III, \S 3, no. 4, Proposition 7]{bourbaki2013general} that the group completion $ \overline{G}$ is Polish.
    By assumption, $\iota(U)$ is closed in $\overline{G}$, making $\iota(U)$ a Polish group.
    The completion map $\iota$ is a homeomorphism onto the image, making $\iota|_U\colon U\to \iota(U)$ a homeomorphism from $U$ to $\iota(U)$.
    Hence, with respect to the $\w$-subspace topology, $U$ is a Polish group.
    Following the Banach Open Mapping Theorem~\ref{thm : banach open mapping theorem}, restricted to $U$, the identity map onto the weaker topology is open, hence the Polish topology $\w_U$ is equal to the original group topology on $U$.
\end{proof}
\begin{cor}
\label{cor : minimal has closed image}
    Let $G$ be a Polish group. Then $G$ is minimal if and only if the image of $G$ under any injective continuous homomorphism is closed.
\end{cor}
\begin{lemma}[{\cite[Theorem 7.2.5]{MR1015288}}]
\label{lem : minimal has compact center}
    Let $G$ be a Polish group. If $G$ is minimal, then its center, $Z(G)$, is compact.
\end{lemma}
\begin{proof}
    Denote by $Z=Z(G)$ the center of $G$. Let $H$ be a topological group and $f\colon Z \to H$ a continuous injective homomorphism with a dense image.
    We will show that $f(Z)$ is  closed in $H$.
    
    Since $Z$ is central, $\Delta_Z=\{(z,f(z))\in G\times H :\  z\in Z \}$, the graph of the function $f$, is a closed normal subgroup of $G\times H$. 
    Let $\hat{f}\colon G\to  (G\times H)/\Delta_Z$,  be the quotient map restricted to $G\times \{e\}$, $\hat{f}(g)=(g,e_H)\Delta_Z$, and let $\delta\colon H\to (G\times H)/\Delta_Z $ be the quotient map restricted to $\{e\}\times H$, $\delta(h)=(e_G,h)\Delta_Z$. 
    Observe that $\hat{f}$ is injective, so by minimality of $G$, $\hat{f}(G)$ is closed. From the following commuting diagram,
    \[
    \begin{tikzcd}
        Z \arrow[r,"f"] \arrow[d,hookrightarrow] & H \arrow[d,"\delta"] \\
        G \arrow[r,"\hat{f}"] & (G\times H)/\Delta_Z
    \end{tikzcd}
    \]
    we get that $f(Z)=\delta^{-1}(\hat{f}(G) )$, and as $\hat{f}$ and $\delta$ are continuous, $f(Z)$ is a pre-image of the closed subset $\hat{f}(Z)$, hence $f(Z)$ closed.
    We get that the center $Z$ is minimal therefore by Theorem~\ref{thm : Polish/LCA not minimal} the group $Z$ is compact.
\end{proof}
\begin{lemma}
\label{lem : product of relitive minimal}
    Let $G$ be a Polish group, and let $A_1,A_2\lhd G$ closed normal subgroups that are minimal in $G$, and such that $A_1\cap A_2=\{e\}$.
    Then their product $A_1A_2$ is minimal in $G$.
\end{lemma}
\begin{proof}
	Let $f \colon G \to H$ be an injective continuous homomorphism with a dense image, we will show that $f(A_1A_2)$ is closed.
    For any $N_1,N_2\lhd H$ closed normal subgroups.
	Letting $\rho\colon H\to H/N_1$ be the quotient map, and $\iota\colon H/N_1\to H/N_1 \times H/N_2$ be the embedding of $H/N_1$ in the product $H/N_1 \times H/N_2 $. 
	The maps $\rho$ and $\iota$ are continuous, making 
	\[
	    \psi\coloneqq \iota\circ\rho\colon H\to H/N_1 \times H/N_2, \ \ \psi: h\mapsto (hN_1,eN_2)
	 \]
	a continuous map.

    For $i=1,2$, the subgroup $A_i\lhd G$ is normal, so its image $f(A_i)\lhd f(G)$ is normal in the dense subgroup $f(G)$ since $f(A_i)$ is closed in $H$, we get that $f(A_i)$ is a closed normal  subgroup in $H$.
    Set $N_i=f(A_i)$, and denote by $\pi \colon H\to H/N_1\times H/N_2$ the product of the quotient maps. 
    Consider the continuous map 
    \[
    f_\times\coloneqq \pi \circ f \colon G \to H/N_1\times H/N_2 , \ \ f_\times: g\mapsto (f(g)N_1,f(g)N_2)
    \]
    Since $A_1\cap A_2=\{e\}$, the map $f_\times$ is injective, by relative minimality of $A_2$ in $G$, the subgroup $f_\times(A_2)\le H/N_1\times H/N_2 $ is closed. 
    Therefore $f(A_1A_2)=\psi^{-1}(f_\times(A_2))$ is a preimage of a closed subset, so it is closed.
\end{proof}
\begin{cor} 
    \label{cor : product of relitive minimal}
    		Let $G$ be a topological group, and let $A_i\lhd G$ $i\in \{1,...,n\}$ a collection of normal subgroups that are minimal in $G$ such that $\cap A_i=\{e\}$. Then the product $\prod_i A_i$ is minimal in $G$.
\end{cor}
\subsection{Co-minimality}
Similarly to relative minimality, the reader may benefit from keeping in mind the following example.
\begin{example}
    Consider the group $\bbR^* \ltimes \bbR$ with the standard topology, and the subgroup $\bbR\le \bbR^* \ltimes \bbR $. Although the group $\bbR^*$ is not minimal, it does have the following property (see Lemma~\ref{lemma: baby case co minimality}): any weaker topology on $\bbR^* \ltimes \bbR$ induces the same quotient topology on the quotient $\bbR^* \ltimes \bbR/ \bbR$.
\end{example}
\begin{definition}
    Let $G$ be a topological group.
    A closed subgroup $U\le G$ is said to be \textit{co-minimal in $G$} if any weaker topology on $G$  induces the same quotient topology on the quotient space $G/U$.
\end{definition}

\begin{example}
    Any minimal co-compact subgroup is co-minimal.
\end{example}
\begin{lemma}
\label{lamma : minimal and co-minimal}
    Let $G$ be a topological group, and $U\le G$ a subgroup. If $U$ is minimal in $G$ and co-minimal in $G$, then $G$ is minimal.
\end{lemma}
\begin{proof}
    Let $\s$ denote the given group topology on $G$. For any $\w\le \s$ weaker group topology, $\w_U=\s_U$ and $\w_{G/U}=\s_{G/U}$, hence by Merson's Lemma~\ref{lemma : Mersons lemma}, $\w=\s$.
\end{proof}
\begin{cor}
\label{cor : open compact minimal}
    Let $G$ be a topological group, and $U\le G$ be a compact open subgroup. There does not exist a strictly weaker group topology $\w$ on $G$, such that $U$ is open in $\w$.
\end{cor}
\begin{proof}
    Denote by $\s$ the given group topology on $G$.
    For any $\w\le \s$ weaker group topology, by the compactness of $U$, $\w_U=\s_U$. And since $U$ is open with respect to both $\w$ and $\s$, the quotient topologies are both discrete, hence $\w_{G/U}=\s_{G/U}$, hence by Lemma~\ref{lamma : minimal and co-minimal}, $\w=\s$.
\end{proof}
\subsection{Minimal properties of semi-direct products of groups}
A group topology on a semi-direct product $L\ltimes U$ induces subspace topologies on the subgroups $L$ and $U$, and $L$ acts continuously on $U$ by conjugation.
Conversely, for $L$ and $U$ topological groups with $L$ acting by continuous automorphisms on $U$, the product topology is a topological group structure on $L\ltimes U$.
Here we will emphasize the relation between topological conditions on the action of $L$ on $U$, and the minimality of $U$ as a subgroup of $L\ltimes U$.
\begin{rem}
\label{remark : qutient topology semi direct}
    Assume $L$ acts continuously by automorphisms on $U$. If $C\le L$ is a subgroup in the kernel of the $L$ action then $(L\ltimes U)/C$ with the quotient of the product topology is isomorphic to  $L/C\ltimes U$ with the product of the quotient topology.
    If $N\lhd U$ is an $L$-invariant closed normal subgroup then $(L\ltimes U)/N$ with the quotient of the product topology is isomorphic to $L\ltimes U/N$ with the product of the quotient topology.
\end{rem}
\begin{remark}
\label{remark: stronger action continuous}
    Let $G$ be a topological group acting continuously on a topological space $X$.
    Then $G$ acts continuously on $X$ with respect to any stronger topology on $G$.
\end{remark}
\begin{lemma}[{\cite[Proposition 4.4]{MR2556080}}]
\label{lem : Action is minimal- semi direct product minimal}
    Let $L$ and $U$ be topological groups such that $L$ acts continuously by automorphisms on $U$. 
    Then $U$ does not admit a strictly weaker group topology such that the action of $L$ on $U$ is continuous, if and only if $U$ is minimal in $L\ltimes U$.
\end{lemma}
\begin{proof}
Denote by $\s_L$ and $\s_U$ be the given group topologies on $L$ and $U$ respectively.

    Assume $\w$ is a group topology, weaker than the product group topology on $L\ltimes U$.
    Since $\w$ is a group topology, the subgroup $L$ with the $\w$-subspace topology, $\w_{L}$, acts continuously by conjugation on $U$ with respect to the $\w$-subspace topology, $\w_U$.
    By Remark~\ref{remark: stronger action continuous}, embellished with the stronger $\s_{L}$ group topology, the group $L$ acts continuously by conjugation on $U$ with respect to $\w$-subspace topology, $\w_U$ on $U$
    Therefore if $U$ does not admit a strictly weaker group topology such that the action of $L$ on $U$ is continuous we get that $\w_U=\s_U$, hence $U$ is minimal in $L\ltimes U$. 
    
    Conversely, assume $U$ is minimal in $L\ltimes U$. If $U$ admits a weaker group topology $\w_U$, such that the action of $L$ on $U$ is continuous is minimal in $G$. The product topology $\s_L\times \w_U$ is a group topology on $L\ltimes U$, and by relative minimality of $U$, $\w_U=\s_U$ as needed.
\end{proof}

\begin{cor}
\label{cor : minimality under qutient}
    Let $G=L\ltimes U$, and let $C\lhd L$ be a closed normal subgroup acting trivially on $U$. 
    Then $U$ is minimal in $G$, if and only if $U$ is minimal in $G/C\simeq L/C\ltimes U$.
\end{cor}
\begin{proof}
Let $\pi \colon L\to L/C$ be the quotient map. Following the commutative diagram,
    \[
    \begin{tikzcd}
        L\times U \arrow[rd,"\alpha"] \arrow[d,"\pi"]& \\
        L/C\times U \arrow[r,"\alpha_c"] & U
    \end{tikzcd}
    \]
    and since $\pi$ is open and continuous, the map $\alpha$ is continuous if and only if $\alpha_c$ is continuous.
    We get that $L$ acts continuously on $U$, if and only if $L/C$ acts continuously on $U$. Hence by Lemma~\ref{lem : Action is minimal- semi direct product minimal}, $U$ is minimal in $G$, if and only if $U$ is minimal in $G/C=L/C\ltimes U$.
\end{proof}

\subsection{Sealed groups}
We saw that Polish minimal groups have the property that any continuous injective homomorphism mapping them to another topological group has a closed image. 
In this subsection, we are interested in groups having a closed image under any continuous homomorphism.
\begin{definition}
\label{def : sealed}
    Let $G$ be a topological group. We say that $G$ is \textit{sealed} if for any continuous homomorphism $f\colon G\to H$ into a topological group $H$, the image $f(G)$ is closed. 
\end{definition}
\begin{example}
    Compact groups are sealed.
\end{example}
\begin{example}
    Simple Polish minimal groups are sealed, as any homomorphism from a simple group is either trivial or injective.
\end{example}
Also here, the relative property is very useful.
\begin{definition} \label{def : selaed pair}
    For a subgroup $U\le G$, we say that \textit{$U$ is sealed in $G$}, if for any continuous homomorphism $f\colon G\to H$ into a topological group $H$, the image $f(U)\le H$ is closed. 
\end{definition}
\begin{lemma}[Cf. {\cite[Proposition 4.2]{dikranjan1998categorically}}]
\label{lem : sealed 3 space prop}
    Let $N\lhd G$ be a closed normal subgroup, with the quotient map $\pi\colon G\to G/N$. 
    If $N$ is sealed in $G$ and $U\le G$ is a subgroup such that $\pi(U)$ is sealed in $\pi(G)$, then $UN$ is sealed in $G$. 
\end{lemma}
\begin{proof}
    Let $f\colon G \to H$ be a continuous homomorphism with a dense image. Since the image of $N$ is closed, we get that $f(N)$ is a closed normal subgroup of $H$. Therefore, the induced map $\hat{f}\colon G/N \to H/f(N)$ is a morphism of topological groups. Since $\pi(U)$ is sealed in $\pi(G)$, we get that $\hat{f}(\pi(U))$ is a closed subgroup in $H/f(N)$. Hence Its pre-image in $H$, which is equal to $f(UN)$, is closed.
\end{proof}
\begin{cor}
\label{cor : co compact sealed}
    Let $N\le G$ be a co-compact subgroup. If $N$ is sealed, then $G$ is sealed.
\end{cor}
\begin{proof}
    Since compact groups are sealed, the proof follows from Lemma~\ref{lem : sealed 3 space prop} above, by taking $U=G$.
\end{proof}
\begin{cor} 
\label{cor : product of normal sealed}
    If $U_1,U_2$ are normal subgroups that are sealed in $G$ then $U_1U_2$ is sealed in $G$.
\end{cor}
\begin{proof}
    Let $\pi \colon G \to G/U_2$ be the quotient map. Given a topological group $H$, and continuous homomorphism $f\colon G/U_2 \to H$, the map $f\circ \pi \colon G\to H$ is a continuous homomorphism. Since $U_1$ is sealed in $G$, $f(\pi(U_1))$ is closed in $H$. Hence $\pi(U_1)$ is sealed in $\pi(G)$, so by Lemma~\ref{lem : sealed 3 space prop}, $U_1\cdot U_2$ is sealed in $G$.
\end{proof}
\begin{cor}[{\cite[Theorem 2.13]{dikranjan1998categorically}}]
\label{cor : product of sealed}
    If $G_1,G_2$ are sealed groups, then so is $G_1\times G_2$.
\end{cor}
\section{Rigidity of scalar multiplication}
\label{Rigidity of scalar multiplication}

Let $k$ be a local field with absolute  value $| \cdot|$, and let $V$ be a $n$-dimensional $k$-vector space with norm $\| \cdot \|$.
We consider the topology on $k$ to be the locally compact field topology, and the topology on $V$ to be the locally compact norm topology.

Although the norm topology on $V$ is not minimal (see Theorem~\ref{thm : Polish/LCA not minimal}), we saw in Lemma~\ref{lemma : baby case multiplication by ball} that amongst all additive group topologies on $V$ admitting a continuous scalar multiplication by elements in an open neighborhood of $0$, the norm topology on $V$ is minimal.

The main goal of this section is to prove Lemma~\ref{lem : main result scalar multiplication} below. The lemma asserts that if the additive group $V$ admits a continuous scalar multiplication restricted to an analytic multiplicative subgroup then the group topology on $V$ is minimal.
\begin{lemma}
\label{lem : main result scalar multiplication}
    Let $k_0$ be a local field such that $k$ is a finite filed extension of $k_0$, and let $S\le k^*$ be an infinite $k_0$-analytic subgroup. 
    Then $V$ does not admit an additive group topology, which is strictly weaker than the norm topology, and such that the $S$ action on $V$, given by scalar multiplication, is continuous.

    Moreover, given an $S$-invariant compact subgroup $N\le V$. The topological group $V/N$ does not admit group topology that is strictly weaker than the quotient topology induced from $V$, and such that the induced action of $S$ on $V/N$ is continuous.
\end{lemma}
\subsection{Moderated subsets}
In order to prove Lemma~\ref{lem : main result scalar multiplication} we will need the notion of a moderated set. We will see in Lemma~\ref{lem : multiplication with moderated set} that admitting a continuous multiplication over a moderated subset guarantees that the analytic topology on the vector space is minimal.
\begin{definition}
\label{defn : moderated set}
    Let $A\subset k$ be a subset containing $0$. We call the set $A$ \textit{moderated} if $A$ contains a null sequence $(a_n)_{n\in \bbN}$, i.e. $\lim_{n} \|a_n\|= 0$, such that $1<|a_n|\cdot |a_{n+1}|^{-1}\le \lambda$ for some $\lambda\in \bbR^+$.
\end{definition}
\begin{example}
    Any 0-neighborhood in $k$ is moderated.
\end{example}
\begin{example}
\label{exmpale : SO(2)-1 moderated}
    Let $k=\bbC$ and $A=S^1-1$ where $S^1$ is the unit circle in $\bbC$. For any value $\delta<1$, the set $A$ contains an element $a$ with absolute value $\delta$, so $A$ is moderated.
\end{example}
Example~\ref{exmpale : SO(2)-1 moderated} can be viewed as a special case of Lemma~\ref{lemma: analytic subset moderated} below, observing that the set $S^1-1$ is an $\bbR$ analytic subset of $\bbC$. 

Assume $k$ is a degree $m$ extension of a local field $k_0$. We may view $k\simeq k_0^m$ as an $m$-dimensional $k_0$-analytic manifold. Lemma~\ref{lemma: analytic subset moderated} asserts us that $k_0$-analytic subsets are moderated.

\begin{lemma}
\label{lemma: analytic subset moderated}
    Assume that $k$ is a finite extension of a local field $k_0$, and let $A\subset k$ be a positive dimensional $k_0$-analytic sub-manifold containing $0$, then $A$ is a moderated subset.
\end{lemma}
\begin{proof}
    Since $A$ is a positive dimensional analytic sub-manifold, there exists an analytic immersion $f\colon B \to A$ satisfying $f(0)=0$, where $B$ is the open unit ball in $k_0$.
    We let $L$ be the derivative of $f$ at $0$ and denote $a= \|L(1)\|$.
    We fix a non zero $x\in B$ and observe that for every $n$,
    $\|L(x^n)\|=|x^n|\cdot a$.
    We fix $\epsilon<a$ small enough so $(a -\e)(a+\e)^{-1}> |x|$ and let $\lambda=(a-\e)^{-1}(a+\e)|x|^{-1}$.
    By definition of the differential, there exists $N\in \bbN$ such that for $n\geq N$,
    \[
        -\e <\frac{\|f(x^n)\|}{|x^n|}-\frac{\|L(x^n)\|}{|x^n|} <\e,
    \]
    thus 
    \[
        0<a -\e <\frac{\|f(x^n)\|}{|x^n|} <a+\e.
    \]
    We set $a_n=f(x^{n+N})\in A$ and note that for every $n$
    \[ \frac{|a_n|}{|a_{n+1}|}=\frac{\|f(x^{n+N})\|}{|x^{n+N}|} \cdot \frac{|x^{n+N+1}|}{\|f(x^{n+N+1})\|}\cdot |x|^{-1}, \]
    thus 
    \[ 1< (a -\e)(a+\e)^{-1}|x|^{-1} < \frac{|a_n|}{|a_{n+1}|} < (a-\e)^{-1}(a+\e)|x|^{-1}= \lambda. \]
    This finishes the proof.
\end{proof}
\begin{lemma}
\label{lem : one minus moderated set}
    Let $S\subset k$ be a subset such that the scalar multiplication map
        $$m_0 \colon S\times V\to V, \ \ (s,v)\mapsto s\cdot v$$ 
    is continuous with respect to an additive group topology on $V$. Then the scalar multiplication 
        $$m_1 \colon (S-1)\times V\to V, \ \ (s-1,v)\mapsto (s-1)\cdot v$$
    is continuous.
\end{lemma}
\begin{proof} 
    The map $ \mathrm{+_1}\colon k \to k$ sending $a$ to $a+1$ is continuous, maps the subset $S-1$ to the subset $S$.
    By continuity of $m_0$ the map 
        $$\alpha\colon S\times V \to V\times V, \ \ (s,v)\mapsto (s\cdot v,v) $$
    is continuous.
    The subtraction of vectors is a continuous operation on $V$.
    Denoting the subtraction map by $ \mathrm{-}\colon V \times V \to V$, we get continuity of the the multiplication map, $m_1\coloneqq  \mathrm{-} \circ  \alpha \circ (  \mathrm{+_1}\times Id) $,
    \[
      m_1 \colon (S-1)\times V \to V, \ \  (s-1,v)\mapsto s\cdot v-v=(s-1)\cdot v
    \]
    as needed.
\end{proof}

\subsection{Weaker topologies on quotients of vector spaces}
Similar to the assertion of Lemma~\ref{lemma :  weaker topology on V unbounded open sets},  weaker topologies on quotients of vector spaces can be classified as those having no bounded open subsets.
\begin{lemma}
\label{lemma: weaker top not bounded }
    Let $N\le V$ be a compact additive subgroup. Then any strictly weaker topology on $V/N$ does not have any open sets with bounded pre-image in $V$. 
\end{lemma}
\begin{proof}
    If $k$ is Archimedean, $V$ does not admit any non-trivial compact subgroups and the claim follows from Lemma~\ref{lemma :  weaker topology on V unbounded open sets}. 
    So we may assume that $k$ is non-Archimedean. Denote by $\pi\colon V\to V/N$ the quotient map. Assume by contradiction that there exist a group topology $\w$, which is a strictly weaker topology on $V/N$ admitting a $0_{V/N}$-neighborhood $W\in \w$, such that $\|\pi^{-1}(W)\|<\lambda$.
    The norm closed ball, $\overline{B}_\lambda\coloneqq \{x\in V :\ |x|\le \lambda \}$, is a compact subgroup of $V$ and therefore $\pi(\overline{B}_\lambda)$ is compact subgroup in $V/N$. Observe that $W\subset \pi(\overline{B}_\lambda)$, thus $\pi(\overline{B}_\lambda)$ is $\w$-open. Hence by Corollary~\ref{cor : open compact minimal} the $\w$ topology is equal to the quotient of the norm topology.
\end{proof}
We will see that, given a moderated subset in $k$, 
the norm topology is minimal
amongst the additive group topologies admitting a continuous multiplication by this moderated set.

\begin{lemma}
\label{lem : multiplication with moderated set}
    Let $A\subset k$ be a moderated subset, and let $m\colon A\times V \to V$ be the scalar multiplication map. Then there does not exist an additive group topology on $V$, which is strictly weaker than the norm topology. and such that $m$ is a continuous map.
    
    Moreover, given an $A$-invariant compact subgroup $N\le V$. The topological group $V/N$ does not admit a group topology that is strictly weaker than the quotient topology induced from $V$, and such that the induced scalar multiplication $m'\colon A\times V/N \to V/N$ is a continuous map.
\end{lemma}
Reading the proof, the reader should have in mind the proof of Lemma~\ref{lemma : baby case multiplication by ball}, which is a special case. 

\begin{proof}
    We will show that any weaker topology on $V/N$ admitting a continuous $A$-multiplication must have an open subset with a bounded pre-image in $V$. The statement then follows from Lemma~\ref{lemma: weaker top not bounded }.
    
    Let $\w$ be a weaker topology on $V/N$ such that the map
        $$m'\colon A\times V/N\to V/N, \ \ (a,v+N)\mapsto a\cdot v +N$$ 
    is continuous and assume by contradiction that all $\w$ open set in $V/N$ have unbounded pre-images in $V$.
    Since $A$ is a moderated set there exists $\lambda\in \bbR^+$ and a null sequence $\{a_n\}_{n\in \bbN}\subset k$, with $1< |a_n|\cdot |a_{n+1}|^{-1}\le\lambda$.
    Let $\pi \colon V\to V/N$ be the quotient map. Since $N$ is compact there exist $\beta\in \bbR^+$ such that $\|N\|<\beta$.
    The annuls $\ann{(\beta,\beta \lambda)}\coloneqq \overline{B}_{\beta \lambda}\setminus B_\beta\subset V$ is a compact subset in the norm topology on $V$, and does not intersect $N$.
    Therefore, the image of the annulus under $\pi$, $\pi(\ann{(\beta,\beta \lambda)})$, is a compact set not containing the identity $0_{V/N}$. 
    Let $W$ be an $\w$ open neighborhood separating $0_{V/N}$ from $\pi(\ann{(\beta,\beta \lambda)})$.
    By continuity of $m'$, there exist a $\w$ open $0_{V/N}$ neighborhood $W'$, and a open $0$ neighborhood $A_0\subset A$ such that $m'(A_0\times W') \subset W$.
    Note that there exists $N$ such that for every $n\geq N$, $a_n\in A_0$.
    Since we assumed by contradiction that $\pi^{-1}(W')$ is unbounded, there exists $v\in \pi^{-1}(W')$ such that 
        $$\|v\|> \beta \cdot|a_N|^{-1}$$
    Since $|a_n|^{-1}$ is a strictly increasing sequence with $\lim_n |a_n|^{-1}= \infty$, there exists $N<n_0$ such that
        $$ \beta \cdot|a_{n_0}|^{-1} <\|v\| \le \beta \cdot|a_{n_0+1}|^{-1}, $$
    thus $a_{n_0}\cdot v\in \ann{(\beta,\beta \lambda)}$ since
    $$
     \beta=\beta \cdot|a_{n_0}|^{-1}\cdot |a_{n_0}|<\|a_{n_0}\cdot v\|\le  \beta \cdot|a_{n_0+1}|^{-1}\cdot |a_{n_0}|\le \beta \cdot\lambda
    $$
    Observe that $(a_{n_0},\pi(v)) \in A_0\times W'$. 
    Since $m'(A_0\times W') \subseteq W$, thus $W$ intersects the annulus $\ann{(\beta,\beta \lambda)}$.
    This is a contradiction.
\end{proof}
We are now ready to prove the main lemma~\ref{lem : main result scalar multiplication}.
\begin{proof}[Proof of Lemma~\ref{lem : main result scalar multiplication}]
    Following Lemma~\ref{lem : one minus moderated set}, we get a $S-1$ continuous multiplication on $V$. Moreover, if $N\le V$ in $S$-invariant it is $S-1$-invariant. 
    By Lemma~\ref{lemma: analytic subset moderated}, the $k_0$-analytic set $S-1$ is moderated.
    The lemma now follows from Lemma~\ref{lem : multiplication with moderated set}.
\end{proof}
\section{Linear algebraic action of Tori}
\label{Linear algebraic action of Tori}
We will be using the following consequence of Schur's Lemma. 
\begin{lemma} [Schur's lemma]
\label{lemma :schurs lemma}
    Let $(\pi,V)$ be a finite-dimensional irreducible $ k$-representation of a commutative group $T$. Then the division ring $\End_T(V)$ is commutative and it is a finite field extension of $k$. Moreover, as $\End_T(V)$-modules, $\End_T(V)\cong V$.
\end{lemma}
\begin{proof}
    Since $T$ is commutative, for any $t\in T$ the linear operator $\pi(t)$ is intertwining. We therefore get a map $\pi:T\to \End_T(V)$. 
    Irreducibility of the $T$ representation implies that $V$ is irreducible as an $\End_T(V)$-module.
    Therefore $V$ must be a quotient of $\End_T(V)$, but since $\End_T(V)$ is a division ring we get that $V\cong \End_T(V)$ as $\End_T(V)$-modules.
    By irreducibly, $\pi(T)$ spans $\End_T(V)$, so the finite-dimensional division ring $\End_T(V)$ is commutative, and therefore a finite field extension of $k$.
\end{proof}

In this section we let $k$ be a local field of characteristic zero and $V$ be a finite-dimensional $k$-vector space.
We let $\bfT$ be a $k$-algebraic torus, and denote by $T=\bfT(k)$ its $k$-points. We fix a $k$-rational representation of algebraic groups $\pi\colon T \to \GL(V)$. I.e. $\pi$ is the $k$-point map of a rational $k$-representation. Note that by Zariski density of the $k$-points, if the algebraic representation of $\bfT$ is irreducible then so is the representation of its $k$-point $T$. 

The topology assumed on the $k$-points of a $k$-algebraic group is the $k$-analytic topology, which we will refer to as the analytic topology.
\begin{lemma}
\label{lem : irreducble tori rep minimal}
    Assume that the representation $\pi$ is irreducible. Let $\alpha \colon T\times V\to V $ denote the continuous action $\alpha(t,v)=\pi(t)v$.
    Then the additive group $V$ does not admit a group topology strictly weaker than the norm topology with respect to which $\alpha$ is a continuous action.
    
    Moreover, suppose $N\le V$ is a $T$-invariant compact subgroup. Then the topological group $V/N$ does not admit a group topology that is strictly weaker than the analytic topology and such that the induced action $\alpha'\colon T\times V/N\to V/N$ is a continuous map.
\end{lemma}
\begin{proof}
    By Schur's Lemma above~\ref{lemma :schurs lemma}, $\pi(T)$ is a representation into $K=\End_{T}(V)$, for $K$ a finite field extension of $k$.
    The representation map $\pi\colon T\to \GL(V)$ is a separable algebraic morphism. Therefore,
    it follows by Proposition \cite[Chapter I, Proposition (2.5.4)(ii)]{margulis1991discrete} that $\pi( T)$ is a $k$-analytic multiplicative subgroup of $K^*$. 
    By the Banach Open Mapping Theorem~\ref{thm : banach open mapping theorem}, $\pi$ is open onto its image so by following diagram
    \[
    \begin{tikzcd}
        T\times V \drar["\alpha"] \dar["\pi"]&  \\
        \pi(T)\times V \rar["m"] & V
    \end{tikzcd}
    \]
    the multiplication by $\pi(T)$ is continuous.
    The claim now follows directly from Lemma~\ref{lem : main result scalar multiplication}.
\end{proof}
\begin{cor}
\label{corollary : irreducable tori space is minimal with action}
    Assume that the $T$-representation $(\pi,V)$ has no invariant vectors.
    Let $\alpha \colon T\times V\to V $ denote the continuous action $\alpha(t,v)=\pi(t)v$.
    Then the additive group $V$ does not admit a group topology strictly weaker than the norm topology, and such that $\alpha$ is a continuous action.
    
    Moreover, given a $T$-invariant compact subgroup $N\le V$. The topological group $V/N$ does not admit a group topology that is strictly weaker than the analytic topology, and such that the induced action $\alpha'\colon T\times V/N\to V/N$ is a continuous map.
\end{cor}
\begin{proof}
    By complete reducibility of tori, 
    the vector space $V$ decomposes into finitely many irreducible sub-representations $V=\oplus_{i=1}^m V_i$. 
    Let $N\le V$ be a compact $T$-invariant subgroup, and let $\w$ be a weaker additive group topology on $V/N$ admitting a continuous $T$-action.
    We prove by induction on the length of the representation, $m$, that $\w$ is equal to the analytic topology on $V/N$.
    
    If $m=1$ the claim follows from Lemma~\ref{lem : irreducble tori rep minimal} above. 
    Assume $m>1$, and let $V_1\le V$ be an irreducible sub-representation.
    As analytic groups $V_1/(N\cap V_1)\cong V_1N/N$, therefore the $\w$-subspace topology on $V_1N/N$, induces a weaker group topology on $V_1/(V_1\cap N)$, admitting a continuous $T$-action. 
    Hence by Lemma~\ref{lem : irreducble tori rep minimal} above, the $\w$-subspace topology on $V_1N/N$ is equal to the analytic topology. In particular, $V_1N/N$ is an $\w$-closed subgroup of $V/N$.
    
    Let $\w'$ denote the $\w$-quotient topology on $(V/N)/(V_1N/N)$.
    Observe that as analytic groups $(V/N)/(V_1N/N)\cong V/V_1N\cong (V/V_1)/(V_1N/V_1)$ are isomorphic, therefore $\w'$ induces a weaker topology on $(V/V_1)/(V_1N/V_1)$ which admits a continuous $T$ action.
    Since $V_1N/V_1$ is a $T$-invariant compact subset of $V/V_1$, which is of length $m-1$ and has no invariant vectors, by induction hypothesis $\w'$ must be equal to the analytic topology.
    Hence by applying Merson's Lemma~\ref{lemma : Mersons lemma} to the subgroup $V_1N/N$ we get that the topology $\w$ is equal to the analytic topology on $V/N$.
\end{proof}
\begin{cor}
\label{corollary :  torus action with no invariant vectors makes minimal}
    Assume that the $T$-representation $(\pi,V)$ has no invariant vectors.
    Then for any $T$-invariant compact subgroup $N\le V$,  the subgroup $V/N$ is minimal in $T\ltimes V/N$. 
\end{cor}
\begin{proof}
    By Corollary~\ref{corollary : irreducable tori space is minimal with action}, the topological group $V/N$ does not admit a group topology that is strictly weaker than the analytic topology, and such that the induced action $\alpha'\colon T\times V/N\to V/N$ is a continuous map.
    Hence by Lemma~\ref{lem : Action is minimal- semi direct product minimal}, $V/N$ is minimal in $T\ltimes V/N$.
\end{proof}
\begin{lemma}
\label{lem : torus actoin is minimal}
    Let $C= \ker(\pi)$ be the kernel of the representation map. Then the topological group $T/C$ does not admit a group topology that is strictly weaker than the analytic topology, and such that the action of $T/C$ on $V$ is continuous.
\end{lemma}
\begin{proof}
    Suppose $\w$ is a weaker topology on $T/C$ such that the action on $V$ is continuous.
    Since $\pi$ is a $k$-algebraic morphism, the image $\pi(T)$ is closed in $\GL(V)$.
    Thus, the injective map $\hat{\pi}\colon T/C\to \GL(V)$, is a bijection onto a Polish group.
    Hence by Corollary~\ref{corollary : bijection onto polish group makes minimal} the $\w$-quotient topology on $T/C$ is equal to the quotient of the analytic topology induced from $T$.
\end{proof}
\section{Minimal algebraic groups}
\label{Minimal algebraic groups}

The primary objective of this section is to prove the following.
\begin{thm}
\label{theorem : minimal condition algebraic groups char 0}
    Let $k$ be a local field of characteristic zero, and let $\bfG$ be a connected $k$-algebraic group.
    Then the group $G=\bfG(k)$ is minimal if and only if the center of $G$ is compact.
\end{thm}

The proof of Theorem \ref{theorem : minimal condition algebraic groups char 0} follows similar outlines to the proof Proposition~\ref{prop : baby case minimal}. 
Proposition~\ref{proposition : minimality of unipotent radical} establishes relative minimality of the unipotent radical, generalizing Lemma~\ref{lemma: baby case reletive minimality}, where relative minimality of the unipotent radical $V$ was shown.
Lemma~\ref{lem : lemma for minimality of solvable} provides co-minimality for the special case of solvable groups, generalizing Lemma~\ref{lemma: baby case co minimality}, which showed co-minimality of $V$.
The Borel subgroup $P\le \PGL{2}{k}$, which played a key role in the proof of Proposition~\ref{prop : second baby case minimal}, is replaced here by a minimal parabolic subgroup.

Henceforth in this section, we let $k$ be a local field of characteristic zero. As usual, algebraic groups over $k$ will be identified with their $\bar{k}$ points and will be denoted by boldface letters. Their $k$-points will be denoted by corresponding Roman letters.
For a connected $k$-algebraic group $\bfG$ we denote its unipotent radical by $\Rad{u}{\bfG}$, and the $k$-points of the unipotent radical by $\Rad{u}{G}$. 
A \textit{$k$-algebraic homomorphism} is a homomorphism  $f\colon G\to H$ that is induced from a $k$-morphism $\bfG\to \bfH$ between algebraic groups, by restricting to its $k$-points.
                
\subsection{Relative minimality of the unipotent radical}

We will show relative minimality of the unipotent radical $\Rad{u}{G}$ in $G$.
As a first step, we show that the center of the unipotent radical is relatively minimal.
\begin{lemma}
\label{lem : center is reletivly miniaml}
    Let $\bfG$ be a connected $k$-algebraic group with unipotent radical $\Rad{u}{\bfG}$. If $G$ has a compact center, then the center of the unipotent radical, $Z(\Rad{u}{G})$, is minimal in $G$. 
\end{lemma}
\begin{proof}
    We let $\bfL\ltimes \Rad{u}{\bfG}$ be a Levi decomposition of $\bfG$, thus $\bfL$ is a reductive $k$-group.
    Denoting by $\bfV$ the center of the unipotent radical $\Rad{u}{\bfG}$, we get that $V=\bfV(k)=Z(\Rad{u}{G})$, which is isomorphic to a finite dimensional $k$-vector space, and the action of $L$ on $V$ is $k$- linear, see the discussion in \S\ref{subsec:Ala}.

    Let $\w$ be a weaker group topology on $G$.
    We will construct a bijective $\w_{V}$-continuous homomorphism
    from $V$ onto a Polish group,
    and conclude by Corollary~\ref{corollary : bijection onto polish group makes minimal} that $V$ is minimal in $G$.

    Fix a maximal $k$-torus $\bfT \le \bfL$,
    and consider the corresponding $k$-rational representation $\pi\colon T\to \GL(V)$.
    Denote by $V^T= \{v\in V : \pi(t)v=v\}$ the subspace of $T$-invariant vectors in $V$.
    Since the action of $T$ on $V$ is algebraic, 
    the quotient space $V_T= V/V_T$ is the $k$-points of the algebraic quotient $\bfV_T= \bfV/\bfV^\bfT$.
    We get a $k$-rational representation $\pi_T\colon T\to \GL(V_T)$ such that the $T$ action on $V_T$ is without invariant vectors.
    
    The $\w$-subspace topology on $T\ltimes V$ is a group topology, and with respect to the $\w$-subspace topologies, the torus $T$ acts continuously on $V$.
    Since the action of $T$ is continuous, the invariant subspace $V^T$ is $\w_V$-closed, making the $\w$-quotient topology $\w_{V/V^T}$ on $V_T$ a group topology.
    
    Remark~\ref{remark: stronger action continuous} implies that endowing $T$ with the analytic topology, it still acts continuously on $V$ equipped with the $\w$-subspace topology, $\w_V$.
    Following Corollary~\ref{corollary : irreducable tori space is minimal with action}, the additive group $V_T$ does not admit a weaker topology such that the action of $T$ is continuous, hence the $\w$-quotient topology $\w_{V/V^T}$ is equal to the analytic topology.
    
    Let $\{\bfT_i\}_{i=1}^m$ be a collection of maximal $k$-tori that generate $\bfL$. Let $\Psi$ be the product of the quotient maps from $V$ onto $V_{T_i}$. That is, $\Psi$ is a $k$-morphism of algebraic groups 
    $$\Psi\colon V\to \prod_{i=1}^m V_{T_i}.$$
    By Proposition~\ref{prop : image and kernel of k point maps}, the image $\Psi(V)$ is a Polish subgroup.
    
    For any $i=1,...,m$, the quotient map from $V$ onto $V_{T_i}$ is continuous with respect to the $\w$ induced topology.
    Since the $\w$-quotient typology, $\w_{V/V^{T_i}}$, is equal to the analytic topology, the map $\Psi$ is a $\w_V$-continuous homomorphism onto the Polish group $\Psi(V)$.
    
    We are left to show that $\Psi$ is injective.
    The kernel $\ker(\Psi)=\cap_{i=1}^m V^{T_i}$ is fixed by all $\bfT_i$,
    hence is fixed by $\bfL$, and in particular by $L$.
    Since $\ker(\Psi)\le V$ is central in $\Rad{u}{G}$, it is a central subgroup of $G$.
    As $G$ is assumed to have a compact center, $\ker(\Psi)$ must be compact, thus $\ker(\Psi)$ must be trivial, as it is a vector subspace of $V$.
    We got that the map $\Psi$ is indeed a $\w_V$-continuous bijection  from $V$ onto a Polish space, as needed.
\end{proof}
Next, we observe that for any $k$-group, the relative minimality of the center of its unipotent radical implies the relative minimality of the unipotent radical itself.
\begin{lemma}
\label{lem : nilpotent center minimal} 
    Let $\bfU$ be a connected unipotent $k$-group.
    Let $\w$ be a group topology on $U$, weaker than the analytic topology.
    If the $\w$-subspace topology is equal to the analytic topology on $Z(U)$, then $\w$ is equal to the analytic topology on $U$.
\end{lemma}
\begin{proof}
    Let $\{\bfZ_i\}_{i=1}^n$ denote the upper central series of the group $\bfU$. By Zariski-density, the collection $Z_i=\bfZ_i(k)$ is the upper central series for $U$.
    We show that if the $\w$-subspace topology on $Z_{i}$ is equal to the analytic topology then the $\w$-subspace topology on $Z_{i+1}$ is equal to the analytic topology. 
    
    Indeed, assume $\w_{Z_i}$ is equal to the analytic topology. 
    Let $\rho\colon \bfZ_{i}\to \bfZ_{i}/\bfZ_{i-1} $ be the algebraic quotient map, for any $u\in \bfU$ define the $k$-morphism of algebraic varieties 
    \[
        \mathbf{c}_{u}\colon \bfZ_{i+1}\to \bfZ_{i}, \    \mathbf{c}_u(h)=[h,u].
    \]
    Observe that the composition $\varphi_u=\rho\circ \mathbf{c}_{u}$ is a homomorphism, as for any $x,y\in \bfZ_{i+1}$
    \begin{align*}
    \varphi_u(xy)
    &= [xy,u] \ \  \bfZ_{i-1} \\
    &= xyuy^{-1}(u^{-1}u)x^{-1}u^{-1} \ \  \bfZ_{i-1} \\
     &= x[y,u]ux^{-1}u^{-1} \ \  \bfZ_{i-1} \\
    &= xux^{-1}u^{-1}[y,u] \ \  \bfZ_{i-1} \\
    &=  \varphi_u(x) \varphi_u(y)
    \end{align*}
    So $\varphi_u$ is a $k$-morphism of algebraic groups, we will denote by $\varphi^k_u\colon Z_{i+1}\to  (\bfZ_{i}/\bfZ_{i-1})(k)$ the map between the $k$-points.

    We claim that there exists a finite collection $\{u_j\}_{j=1}^m\in U$ such that $\cap_{j=1}^m\ker (\varphi^k_{u_j})=Z_i$.
  
    First, we claim that it is enough to intersect over the Zariski-dense set $U$ i.e. that 
    $$
    \bigcap_{u\in \bfU} \ker (\varphi_u)=\bigcap_{u\in U} \ker (\varphi_u)
    $$
    Indeed, let $x\in \cap_{u\in U} \ker (\varphi_u) $ then
    $[x,U]\subset \bfZ_{i-1}$, by Zarsiki-density we get that $[x,\bfU]\subset \bfZ_{i-1}$, implying that $x\in \cap_{u\in \bfU} \ker (\varphi_u)$.
    So $\cap_{u\in \bfU} \ker (\varphi_u)=\cap_{u\in U} \ker (\varphi_u)=\bfZ_i$ as claimed.
    
    Since any descending chain of Zariski closed subsets must eventually be constant, there exists a finite collection $\{u_j\}_{j=1}^m\in U$ such that $\cap_{j=1}^m\ker (\varphi_{u_j})=\bfZ_i$. Since $\ker (\varphi_{u_j})(k)=\ker (\varphi^k_{u_j})$, see Proposition~\ref{prop : image and kernel of k point maps}, we get that $\cap_{j=1}^m\ker (\varphi^k_{u_j})=Z_i$. %
    
    The product map 
    $$
    \Psi= \prod_{i=j}^m \varphi^k_{u_j} \colon Z_{i+1}\to (\bfZ_i/\bfZ_{i-1})(k)^m 
    $$
    is $k$-algebraic homomorphism, hence its image $\Psi(Z_{i+1})$ is a closed subgroup with respect to the analytic topology on $(\bfZ_i/\bfZ_{i-1})(k)^m$.
    By construction, $\ker(\Psi)=\cap_{j=1}^m \ker(\varphi_{u_j})=Z_i$, and so the map $\Psi$ factors through a $\w$-quotient continuous bijection map from $Z_{i+1}/Z_i$ onto the Polish group $\Psi(Z_{i+1})$. We get by Corollary~\ref{corollary : bijection onto polish group makes minimal}, that the $\w$-quotient topology on $Z_{i+1}/Z_i$ is equal to the analytic topology, therefore, by applying Merson's Lemma~\ref{lemma : Mersons lemma} to the subgroup $Z_{i}\le Z_{i+1}$ we get that the $\w$-topology on $Z_{i+1}$ is equal to the analytic topology, as needed.

    By assumption, the $\w$-subspace topology on $Z_1$ is equal to the analytic topology, by applying the claim above inductively for $i=1,..,n$ we get that the $\w$ topology in $Z_n=U$ is equal to the analytic topology.
\end{proof}
\begin{remark}
    Note that the proof of Lemma~\ref{lem : nilpotent center minimal} above, did not use any characteristic assumptions on the local field $k$.
\end{remark}
\begin{prop}
\label{proposition : minimality of unipotent radical}
    Let $\bfG$ be a connected $k$-algebraic group with unipotent radical $\Rad{u}{\bfG}$. If $G$ has a compact center, then the unipotent radical $\Rad{u}{G}$, is minimal in $G$.
\end{prop}
\begin{proof}
    Suppose that $G$ has a compact center. By Lemma~\ref{lem : center is reletivly miniaml}, the center of $\Rad{u}{G}$ is minimal, and we conclude by Lemma~\ref{lem : nilpotent center minimal}.
\end{proof}
\subsection{Minimal conditions for solvable groups}
In this subsection, we assume $\bfG$ is connected solvable $k$-algebraic group. Recall that by structure theory of connected solvable groups \cite[Theorem 10.6]{borel2012linear}, $\bfG$ admits a Levi decomposition such that the Levi subgroup is a maximal torus defined over $k$.  
\begin{lemma} 
\label{lem : lemma for minimality of solvable}
    Let $\bfG$ be a connected solvable $k$-algebraic group, and let $\bfT\ltimes \Rad{u}{\bfG}$ be a $k$-Levi decomposition of $\bfG$.
    Denote by $C=\{t\in T : \forall u\in \Rad{u}{G},\ [t,u]=e\}$ the kernel of the conjugation action of $T$ on $\Rad{u}{G}$.
    Then there exist a finite dimensional $k$-vector space $V$, and an algebraic $k$-rational representation $\pi\colon G\to \GL(V)$, with kernel $\ker(\pi)=C\times \Rad{u}{G}$ such that:
    \begin{enumerate}[label = (\alph*)]
        \item For any group topology $\w$, weaker than the analytic topology on $G$, if the $\w$-subspace topology on $\Rad{u}{G}$ equals the analytic topology, then $\pi$ is $\w$-continuous.
        \item For any closed subgroup $C_0\le C$, and for any group topology $\w'$, weaker than the analytic topology on $G/C_0$, if the $\w'$-subspace topology on $\Rad{u}{G}$ equals the analytic topology, then the map induced by on the quotient $\hat{\pi}\colon G/C_0\to \GL(V)$ is $\w'$-continuous.
    \end{enumerate}
\end{lemma}
\begin{proof}
    Let $\{Z_i\}_{i=1}^m$ be the upper central series of the nilpotent group $\Rad{u}{G}$.
    For any $i=1,...,m$, the subgroup $Z_i$ is normal in $G$, and thus we have an algebraic action of $G$ on the finite dimensional $k$-vector space $V_i=Z_i/Z_{i-1}$.
    We define the representation space $V$ by taking the direct sum of the central quotient spaces, $V= \oplus_{i=1}^m V_i$, and thus
    attain a $k$-rational representation $\pi\colon G\to \GL(V)$.
    
    By construction, for any $i=1,...m$, $[\Rad{u}{G},Z_i]\subset Z_{i-1}$ hence $\Rad{u}{G}\in \ker(\pi)$. 
    By definition $[C,Z_i]=\{e\}$, so $C\in \ker(\pi)$ as well, implying  that $C\Rad{u}{G}\subset \ker(\pi)$.
    We will show that $\ker(\pi)\subset C\Rad{u}{G}$.
    Let $x\in \ker(\pi) $, $x=tu$ for some $t\in T$ and $u\in \Rad{u}{G}$, since $u\in \ker(\pi)$ we get that $t\in \ker(\pi)$.
    We prove by induction on $i=1,...,m$ that $[t,Z_i]=\{e\}$.
    For $i=1$ the assertion is obvious. 
    Assume that $[t,Z_{i-1}]=\{e\}$, we will show that $[t,Z_i]=\{e\}$.
    Fix $z_i\in Z_i$. Since $t$ is in the kernel of the action, the unipotent element $[t,z_i]$ is contained in $Z_{i-1}$.
    By induction assumption, the semisimple element $t$, and thus its inverse $t^{-1}$, commutes with $[t,z_i]$. 
    Observe that $t^{-1}\cdot [t,z_i]=z_i t z_i^{-1}$ is semisimple, so by uniqueness of Jordan decomposition $[t,z_i]=e$ as needed.
    We get that $[t,Z_m]=[t,\Rad{u}{G}]=\{e\}$, and hence $t\in C$, and $x\in C \Rad{u}{G}$.
    Thus $\ker(\pi)=C\Rad{u}{G}$.
    
    Since (a) is a special case of (b), obtained by choosing $C_0=\{e\}$, we are left to prove (b).
    Let $C_0$ be a closed subgroup of $C$, and set $G_0= G/C_0=T/C_0\ltimes \Rad{u}{G}$.
    Let $\w'$ be a weaker group on $G_0$, such that the $\w'$-subspace topology on $\Rad{u}{G}$ equals the analytic topology.
    The conjugation action of $G_0$ on $Z_i$ is continuous with respect to the $\w'$-subspace topologies. 
    In turn, the induced action of $G_0$ on $Z_i/Z_{i-1}$ is continuous with respect to the $\w'$-induced topologies. 
    Since $Z_i\le \Rad{u}{G}$ the $\w'$-induced topology in $Z_i$ is analytic and hence the $\w'$-induced topology on $Z_i/Z_{i-1}$ is equal to the standard analytic norm topology on the vector space. 
    By the construction of the map $\pi$, the action of $G$ on $V$ is via the conjugation action of $G$ on $Z_i$, which is $\w'$ continuous. Therefore the action of $G/C_0$ on $V$ via the conjugation action of $G_0$ on $Z_i$ is $\w'$ continuous.
    \end{proof}
\begin{cor}
\label{cor : compact qutient minimality of solvable}
    Let $\bfG$ be a connected solvable $k$-algebraic group, and let $\bfT\ltimes \Rad{u}{\bfG}$ be a $k$-Levi decomposition of $\bfG$.
    Denote by $C=\{t\in T : \forall u\in \Rad{u}{G},\ [t,u]=e\}$ the kernel of the conjugation action of $T$ on $\Rad{u}{G}$.
    Let $C_0\le C$, be a closed co-compact subgroup of $C$.
    Let $\w$ be a group topology, weaker than the analytic topology on $G/C_0$.
    If the $\w$-subspace topology on $\Rad{u}{G}\le G/C_0$ equals the analytic topology, then $\w$ equals the analytic topology on $G/C_0$.
\end{cor}
\begin{proof}
    Denote $G_0= G/C_0= T/C_0\ltimes \Rad{u}{G} $, and let $A=C/C_0\times \Rad{u}{G}\le G_0$.
    By assumption, the $\w$-subspace topology on $\Rad{u}{G}$ is equal to the analytic topology, implying that $\Rad{u}{G}$ is $\w$-closed. Since $C/C_0$ is compact the $\w$-quotient topology on $A/\Rad{u}{G}$ is equal to the analytic quotient topology, hence by Merson's Lemma~\ref{lemma : Mersons lemma} the $\w$-subspace topology on $A$ is equal to the analytic topology.
    In particular, we get that $A$ is $\w$-closed and that the $\w$-quotient topology on $G_0/A$ is a group topology weaker than the analytic quotient topology.
       
    Following Lemma~\ref{lem : lemma for minimality of solvable} above, there exists a $k$-rational representation $\pi\colon G \to \GL(V)$, with kernel $C\times \Rad{u}{G}$, such that the representation map $\hat{\pi}\colon G_0 \to \GL(V)$ is $\w$ continuous.
    
    Since $\pi$ is algebraic the image $\pi(G)$ is closed, hence we get that $\hat{\pi}$ is a $\w$-continuous representation with kernel $A$. This gives us a $\w_{G_0/A}$ continuous injective map from $G_0/A$ onto the Polish group $\pi(G)$.
    By Corollary~\ref{corollary : bijection onto polish group makes minimal} the $\w$-quotient topology, $\w_{G_0/A}$ is equal to the analytic quotient topology on $G_0/A$. Applying  Merson's Lemma~\ref{lemma : Mersons lemma} to the subspace $A$, we get that the $\w$ topology on $G_0$ is equal to the analytic topology.
\end{proof}
\begin{cor}
\label{cor : solavable with good unipotent radical}
    Let $\bfG$ be a connected solvable $k$-algebraic group, and let $\bfT\ltimes \Rad{u}{\bfG}$ be a $k$-Levi decomposition of $\bfG$.
    Denote by $C=\{t\in T : \forall u\in \Rad{u}{G},\ [t,u]=e\}$ the kernel of the conjugation action of $T$ on $\Rad{u}{G}$.
    Let $\w$ be a group topology, weaker than the analytic topology on $G$.
    If $C$ is compact and the $\w$-subspace topology on $\Rad{u}{G}$ equals the analytic topology, then $\w$ is equal to the analytic topology in $G$.
\end{cor}
\begin{proof}
    Since $C$ is compact, this follows by Corollary~\ref{cor : compact qutient minimality of solvable} above, by taking $C_0=\{e\}$.
\end{proof}
We can now show Theorem~\ref{theorem : minimal condition algebraic groups char 0} for the special case of solvable groups.
\begin{prop}
\label{prop : minimality condition for solvable groups}
    Let $\bfG$ a connected solvable $k$-algebraic group. Then $G$ is minimal if and only if the center of $G$ is compact.

    Moreover, if the center of $G$ is compact then for any central subgroup $C_0\le G$, the group $G/C_0$ is minimal.
\end{prop}
\begin{proof}
    If $G$ is minimal then by Lemma~\ref{lem : minimal has compact center}, its center is compact.
    Assume that the center of $G$ is compact, and let $\bfT\ltimes \Rad{u}{\bfG}$ be a $k$-Levi decomposition of $\bfG$.
    Since $G$ is Zariski dense in $\bfG$, $Z(G)\le Z(\bfG)$. Let $z\in Z(G)$ and write $z=tu$ for $t\in T$ and $u\in \Rad{u}{G}$, observe that $[t,u]=e$ hence the decomposition $z=tu$ is a Jordan decomposition of $z\in Z(\bfG)$, implying that $u\in Z(G)$, but since $Z(G)$ is compact $u=e$. Therefore $Z(G)\subset T$, and $Z(G)=\{t\in T : \forall u\in \Rad{u}{G},\   [t,u]=e\}$ is the kernel of the conjugation action of $T$ on $\Rad{u}{G}$. 
    
    Let $\w$ be a group topology on $G$, weaker than the analytic topology. Since $Z(G)$ is compact, by Proposition~\ref{proposition : minimality of unipotent radical}, $\Rad{u}{G}$ is minimal in $G$. We get that $\w_{\Rad{u}{G}}$ is equal to the analytic topology, hence by Corollary~\ref{cor : solavable with good unipotent radical}, the topology $\w$ is equal to the analytic topology on $G$.
    
    Let $C_0\le Z(G)$ be a central subgroup.
    Since $\Rad{u}{G}$ is minimal in $G$, following Corollary~\ref{cor : minimality under qutient}, $\Rad{u}{G}$ is minimal in $G/C_0$.
    Hence by Corollary~\ref{cor : compact qutient minimality of solvable}, $\w'$ is equal to the analytic topology on $G/C_0$.
\end{proof}

\subsection{Minimal algebraic groups}
\begin{lemma}
\label{lemma : minimality of miniaml parabolic in reductive grp}
    Let $\bfG$ be a connected reductive $k$-algebraic group, and let $P\le G$ be a minimal parabolic in $G$. If $G$ has a compact center, then $P$ is a minimal topological group.

    Moreover, for any central $k$-algebraic subgroup $\bfZ_0\le Z(\bfP)$, the group $P/Z_0$ is minimal.
\end{lemma}
\begin{proof}
    By the structure of parabolic subgroups \cite[Proposition 20.6]{borel2012linear}, the minimal parabolic subgroup $P\le G$ is equal to the semi-direct product $\CC_G(T)\ltimes \Rad{u}{P}$, where $\bfT$ is a maximal $k$-split torus, and $\CC_G(T)$ is the centralizer of $T$ in $G$.
    
    Denote by $\bfA$ the solvable $k$-algebraic group $ \bfT\ltimes\Rad{u}{\bfG}$.
    We claim that $Z(A)<Z(P)$.
    Indeed, for $a\in Z(A)$, since $T\subset A$, clearly $a\in \CC_G(T)$ hence $a\in T$, so $[a,\CC_G(T)]=\{e\}$ and thus $a\in Z(P)$.
    Following Lemma~\ref{lem : center of parabolic}, $Z(P)=Z(G)$, and $Z(G)$ is compact, therefore $Z(A)$ is compact.
    Hence, by Proposition~\ref{prop : minimality condition for solvable groups} the group $A$ is minimal. 
    
    Moreover, let $Z_0\le Z(P)$ be a $k$-algebraic subgroup. As $Z(P)=Z(G)$, the group $Z_0$ is central in $G$ and is therefore a central algebraic torus.
    The group $A'=TZ_0\ltimes \Rad{u}{G}$ is equal to the $k$-points of the $k$-algebraic group $\bfA \bfZ_0$. 
    Since $T\subset A'$ for any $a'\in Z(A')$, $a'\in \CC_G(T)$ hence $a\in TZ_0$ and therefore commutes with $\CC_G(T)$ as well, making $a\in Z(P)$. Hence $Z(A)$ is compact, and by Proposition~\ref{prop : minimality condition for solvable groups} the group $A'/Z_0$ is minimal. 
    
    Following \cite[Chapter I, Proposition (2.3.6)]{margulis1991discrete}, the anisotropic group $\CC_G(T)/T$ is compact. Therefore the solvable subgroup $A\le P$ is co-compact, and in particular, $A'$ is co-compact in $P$.
    Hence $A$, ( $A'/Z_0$), is a minimal co-compact subgroup of $P$,  ($P/Z_0$)  thus by Lemma~\ref{lamma : minimal and co-minimal}, $P$ ($P/Z_0$), is minimal.
\end{proof}
Since parabolic subgroups are co-compact, we get the minimality criteria for reductive groups.
\begin{prop}
\label{prop : minimality for reductive}
    Let $\bfG$ be a connected reductive $k$-algebraic group. Then $G$ is minimal if and only if the center of $G$ is compact.

    Moreover, if the center of $G$ is compact then for any finite central subgroup $Z_0\le G$, the group $G/Z_0$ is minimal.
\end{prop}
\begin{proof}
    If $G$ is minimal then, by Lemma~\ref{lem : minimal has compact center}, its center is compact.
    Assume that the center of $G$ is compact. 
    If $Z_0$ is a central subgroup of the reductive group $G$, it is contained in any maximal torus, and in particular $Z_0\le P$. 
    Let $\w$ be a group topology on $G$ (or on $G/Z_0$), weaker than the analytic topology.
    By Lemma~\ref{lemma : minimality of miniaml parabolic in reductive grp}, $\w_P$ (respectively, $\w_{P/Z_0}$) is equal to the analytic topology. In particular $P$ (resp. $P/Z_0$), is $\w$-closed. Since $P$ is co-compact in $G$, the $\w$-quotient topology is Hausdorff and equal to the analytic quotient topology. 
    Hence by Merson's Lemma~\ref{lemma : Mersons lemma} $G$, ($G/Z_0$), is minimal.
\end{proof}

We are now ready to prove Theorem~\ref{theorem : minimal condition algebraic groups char 0} which states that having a compact center is equivalent to minimality for any connected algebraic group over a local field of characteristic zero.
\begin{proof} 
    [Proof of theorem~\ref{theorem : minimal condition algebraic groups char 0}]
    If $G$ is minimal, then by Lemma~\ref{lem : minimal has compact center}, the center of $G$ is compact.
    
    For the other direction, 
    let $\bfL\ltimes \Rad{u}{\bfG}$, be a Levi decomposition of $\bfG$, with $\bfL$ a reductive $k$-group.
    The central subgroup, $Z(\bfL)^\circ$ is a $k$-torus \cite[Proposition 11.21]{borel2012linear}, and we denote by $\bfT\le Z(\bfL)^\circ$ the maximal $k$-split subtorus as in \cite[8.15]{borel2012linear}. 
    Set $\bfA=\bfT\ltimes \Rad{u}{\bfG}\le \bfG$. 
    We will show that $A$ is minimal and co-minimal in $G$.
    
    Let $\w$ be a group topology on $G$, weaker than the analytic topology.
    Let $C=\{t\in T : \forall u\in \Rad{u}{G},\ [t,u]=e\}$, since $T\le Z(L)$, $C$ is a central subgroup of $G$ and hence compact.
    Moreover, by Proposition~\ref{proposition : minimality of unipotent radical}, the $\w$-subspace topology on $\Rad{u}{G}$ is equal to the analytic topology. Hence by Corollary~\ref{cor : solavable with good unipotent radical}, the $\w$-subspace topology on $A$ is equal to the analytic topology, i.e. $A$ is relatively minimal in $G$.
    
    In order to show co-minimality of $A$, we will show that the group $G/A$ is minimal. Observe that since $\bfA$ is a $k$-split solvable group then $(\bfG/\bfA)(k)\cong G/A$ therefore, we may regard $G/A$ as $k$-points of the reductive $k$-algebraic group $\bfG/\bfA$. We claim that the center of $L/T\cong G/A$ is compact.
    
    Let $\pi\colon \bfL \to \bfL/\bfT$ be the $k$-algebraic quotient map.
    We claim that $\pi(Z(\bfL))=Z(\bfL/\bfT)$, clearly $\pi(Z(\bfL))\subset Z(\bfL/\bfT)$.
    For the other direction, let $s\in \bfL$ be an element that is mapped into $Z(\bfL/\bfT)$, define the algebraic map 
    $$\mathbf{c}_s\colon \bfL \to \bfL, \ \ \mathbf{c}_s\colon l\mapsto [s,l]$$
    Observe that $\mathbf{c}_s(\bfL)\subset \mathcal{D}(\bfL)\cap \bfT $ and since $\bfL$ is connected, $\mathbf{c}_s(\bfL)$ is connected.
    Following \cite[Proposition 14.2]{borel2012linear}, the intersection $ \mathcal{D}(\bfL)\cap \bfT $ is finite, hence by connectedness $\mathbf{c}_s(\bfL)=\{e\}$, so $s\in Z(\bfL)$ as needed.
    By restricting the map $\pi$ to the central subgroup, we get a quotient map $\pi\colon Z(\bfL)\to Z(\bfL/\bfT)$.
    By \cite[Chapter I, Proposition (1.4.5)(i)]{margulis1991discrete}, maximal $k$-split tori in $Z(\bfL/\bfT)$ are images of maximal $k$-split tori in $Z(\bfL)$.
    By the fact that $\pi$ maps $\bfT$ to $\{e\}$, we get that $Z(\bfL/\bfT)$ has no maximal $k$-split torus, implying that $Z(L/T)$ is compact.
    Hence by Proposition~\ref{prop : minimality for reductive}, the group $L/T$ is minimal.
    Finally, by applying Meson's Lemma~\ref{lemma : Mersons lemma} to $A<G$, since the $\w$-subspace topology on $A$ is analytic and $G/A\cong L/T$ is minimal, we get that the $\w$ topology on $G$ is equal to the analytic topology on $G$.
\end{proof}
\section{Sealed algebraic groups}
\label{Sealed algebraic groups}

We saw that for algebraic groups over local fields of characteristic zero, having a compact center is a sufficient condition for minimality. 
Mayer showed in \cite{MR1428116} that 
for connected locally compact groups this necessary condition is in fact sufficient - a connected locally compact group such that all of its quotients have compact centers is sealed.
The following is a generalization that applies to connected algebraic groups over an arbitrary local field of characteristic zero.
\begin{thm}
\label{theorem : sealed theorem}
Let $k$ be a local field of characteristic $0$ and let $\bfG$ be a connected $k$-group with a $k$-Levi decomposition $\bfL \ltimes\Rad{u}{\bfG}$. The following are equivalent
\begin{enumerate}
    \item The group $\bfG(k)$ is sealed.
    \item For any normal unipotent $k$-algebraic subgroup $\bfN\lhd \bfG$ the center $Z((\bfG/\bfN) (k))$ is compact.
    \item The center of the reductive group $\bfL$ is anisotropic (equivalently, the center of $\bfL(k)$ is compact) and the action of $\bfL$ on $\Rad{u}{\bfG}$ is with no non-trivial fixed points, 
    \item The center of the reductive group $\bfL$ is anisotropic (equivalently, the center of $\bfL(k)$ is compact) and the action of $\bfL$ on $\Lie(\Rad{u}{\bfG})$ is with no invariant vectors.
\end{enumerate}
\end{thm}
\begin{remark}
    One can define condition (2') to be the condition that the center of any quotient group is compact, note that clearly (2') implies (2).
    For connected real algebraic groups, the equivalence of (2') and (1) is shown in \cite[Proposition 2.14]{MR1428116}
    and the equivalence of (4) and (1) is shown in \cite[Theorem 2.5]{MR1428116}.
    In both cases the methods of proof rely on the result of Omori, and they differ from the methods used here.
\end{remark}

In this section we assume that $k$ is a local field of characteristic zero.  
As usual, algebraic groups over $k$ will be identified with their $\bar{k}$ points and will be denoted by boldface letters. Their $k$-points will be denoted by corresponding Roman letters.

We let $\bfG$ a connected $k$-algebraic group and denote its unipotent radical by $\Rad{u}{\bfG}$, the $k$-points of the unipotent radical will be denoted by $\Rad{u}{G}$. 
We recall that the group $G^+$ is the subgroup of $G$ generated by the unipotent element in $G$, see \cite[Chapter I, (1.5.2)]{margulis1991discrete}.

\subsection{Sealed reductive groups}
\label{subsecting sealed reductive}
In this subsection, we will consider the special case where $\bfG$ is a connected reductive $k$-group. 

Since the sealed property is closed under taking products and quotients and has the three spaces property (see Lemma~\ref{lem : sealed 3 space prop}),
the sealed property for reductive groups follows by analyzing the case for almost $k$-simple groups.
The following Lemma is a version of Proposition~\ref{prop : baby case sealed}.

\begin{lemma}
\label{lem : simple groups are sealed}
    Let $\bfG$ be a connected almost $k$-simple group. Then $G$ is sealed.
\end{lemma}
\begin{proof}
    Let $f\colon G\to H$, be a continuous homomorphism into a topological group $H$.
    By 
    \cite[Chapter I, Theorem (1.5.6)(i) and Theorem (2.3.1)(b)]{margulis1991discrete}, any non-trivial normal subgroup is either co-compact or central and finite.
    If the kernel of $f$ is co-compact then the image $f(G)$ is compact and hence closed.
    Thus we may assume that $\ker(f)\le G$ is a finite central subgroup. Let $\hat{f}\colon  G/\ker(f) \to H$ be the injective map induced by $f$.
    By Proposition~\ref{prop : minimality for reductive} the group $G/\ker(f)$ is minimal hence $f(G)=\hat{f}( G/\ker(f) )$ is closed in $H$.
\end{proof}
We saw in Corollary~\ref{cor : product of sealed}, that a product of sealed groups is sealed. Thus, as semisimple groups are a quotient of a product of $k$-almost simple groups, see \cite[Theorem 22.10]{borel2012linear}, 
we obtain the following theorem, which was proved by Bader-Gelander in \cite{MR3692904} using a different method.
\begin{cor}[{\cite[Theorem 5.1]{MR3692904}}]
\label{corollary : semi simple groups are sealed}
    Let $\bfG$ be a connected semisimple $k$-group, then $G$ is sealed.
\end{cor}
\begin{lemma}
\label{lemma : G+ is seald}
    Let $\bfG$ be a connected semisimple $k$-group, then $G^+$ is sealed.
\end{lemma}
\begin{proof}
    Assume without loss of generality that $\bfG$ has no anisotropic factors. 
    Let $\tilde{\bfG}$ be the simply connected covering of $\bfG$.
    Following \cite[Chapter I, Theorem (2.3.1)(a)]{margulis1991discrete},  $\tilde{\bfG}(k)=\tilde{G}^+$, so by Corollary~\ref{corollary : semi simple groups are sealed}, $\tilde{G}^+$ is sealed. 
    Since the quotient map, $f\colon \tilde{\bfG}\to \bfG$ is a central $k$ isogeny, we have by \cite[Chapter I, Proposition  (1.5.5)]{margulis1991discrete} that $f(\tilde{\bfG}(k)^+)=G^+$, so $G^+$ is sealed as it is a quotient of a sealed group.
\end{proof}
\begin{cor}
\label{corollary : finite index in ss}
Let $\bfG$ be a connected semisimple $k$-group. Then any finite index subgroup of $G$ is sealed.
\end{cor}
\begin{proof}
Let $H$ be a finite index subgroup of $G$.  The subgroup $G^+$ is co-compact and is contained in any subgroup of finite index (see \cite[Chapter I, Corollary (1.5.7) and Theorem (2.3.1)(b)]{margulis1991discrete}). Since $G^+$ is sealed and co-compact in $H$, by Corollary~\ref{cor : co compact sealed}, $H$ is sealed.
\end{proof}
\begin{prop}
\label{prop : reductive sealed}
Let $\bfG$ be a connected reductive $k$ group, then $G$ is sealed if and only if $G$ has a compact center.
\end{prop}
\begin{proof}
Following \cite[Proposition 14.2]{borel2012linear}, the solvable radical $\Rad{s}{\bfG}=Z(\bfG)^\circ$ is a $k$-torus.
Since $Z(G)=Z(\bfG)(k)$ is compact we get that the solvable radical $\Rad{s}{G}$ is a sealed normal subgroup.
The algebraic $k$-group $\bfG/\Rad{s}{\bfG}$ is a connected semisimple $k$-group, and the quotient group $G/\Rad{s}{G}$ is isomorphic to a finite index subgroup of $\bfG/\Rad{s}{\bfG}$ \cite[Proposition 6.14]{platonov1993algebraic}. 

Hence by Corollary~\ref{corollary : finite index in ss}, $G/\Rad{s}{G}$ is sealed.
We get by the three spaces property for sealed groups, Lemma~\ref{lem : sealed 3 space prop}, that the group $G$ is sealed. 
\end{proof}
%
%

%
\begin{lemma}
\label{lemma : normal subgroup of reductive compact ceter}
    Let $\bfG$ be a connected reductive $k$ group such that $G$ has a compact center. Then for any normal connected $k$ group $\bfN\lhd \bfG$, the center of $N$ is compact.
\end{lemma}
\begin{proof}
    Let $\bfZ\le \bfG$ be the center of $\bfG$, the group $\bfG/\bfZ$ is semisimple and so the image of $\bfN$ under the central quotient map is a normal subgroup of a semisimple group so is semisimple (see \cite[Proposition14.10]{borel2012linear}). Thus $Z(N/Z)$ is compact, hence by compactness of $Z$,  $Z(N)$ is compact.
\end{proof}
\begin{cor}
    \label{corollary : normal subgruop of reductive sealed is sealed }
    Let $\bfG$ be a connected reductive $k$-group with a compact center. Then for any $k$-algebraic normal subgroup $\bfN$, the group $N=\bfN(k)$ is sealed.
\end{cor}

\subsection{Proof of Theorem~\ref{theorem :  sealed theorem}}
Ultimately, we will prove Theorem~\ref{theorem :  sealed theorem} by induction on the nilpotency degree of the unipotent radical.
First, we establish a few helpful results for the case where the unipotent radical is abelian. 
\begin{lemma}
\label{lem : selaed open subgroup unipotent radical}
    Let $k=\bbQ_p$, and let $\bfG$ be a connected $k$-algebraic group with a Levi decomposition $\bfL\ltimes \Rad{u}{\bfG}$ over $k$.
    Assume that $\Rad{u}{\bfG}$ is commutative, and that $L$ acts on $\Rad{u}{G}$ with no invariant vectors.
    Let $N\lhd \Rad{u}{G}$ be an $L$ invariant compact open subgroup. Then $\Rad{u}{G}/N$ is minimal in $L\ltimes \Rad{u}{G}/N$
\end{lemma}
\begin{proof}
    The abelian group $\Rad{u}{G}$ is isomorphic to a finite-dimensional $k$-vector space, which will be denoted $V=\Rad{u}{G}$.
    Let $\bfT$ be a $k$-maximal torus in $\bfL$, let $(V/N)^T$ be the subgroup of $T$-invariant elements in $V/N$, and denote by $(V/N)_T$ the co-invariant quotient space.
    Clearly the set of $T$-invariant vectors in $V$, $V^T$, is sent to the set of $T$-invariant vectors in the quotient space $V/N$. Hence the quotient map $V\to (V/N)_T$ factors via the quotient space $V_T$. 
    Let $M_T\le V_T$ be the kernel of the corresponding map $V_T \to (V/N)_T$.
    As $N<V$ is open, $V_T/M_T\cong (V/N)_T$ is an isomorphism of discrete groups.
    It follows by Corollary~\ref{corollary : irreducable tori space is minimal with action} that there is no topology, weaker than the discrete on $V_T/M_T$ and correspondingly $(V/N)_T$ under which the analytic group $T$ acts continuously.

    Let $\w$ be a group topology on $L\ltimes V/N$, weaker than the analytic topology. 
    With respect to the $\w$-subspace topology, $L$ acts continuously on $V/N$. 
    Let $\w'$ denote the $\w$-subspace topology on $V/N$.
    Following Remark~\ref{remark: stronger action continuous}, furnished with the analytic topology, the group $L$ acts continuously on $V/N$, the latter endowed with the $\w'$ topology. 
    We will show that $\w'$ is a discrete group topology by showing the existence of a finite open subgroup.

    The analytic subgroup $T\le L$ acts continuously on $V/N$ with respect to the $\w'$-topology, 
    thus the $T$-invariant subgroup, $(V/N)^T$, is $\w'$-closed. 
    Therefore, the $\w'$-quotient topology on $(V/N)_T$ is a group topology weaker than the analytic topology, admitting a continuous $T$-action.
    By the considerations above, the $\w'$-quotient topology on $(V/N)_T$ is the discrete topology.
    Therefore, the subgroup $(V/N)^T$ is $\w'$-open in $V/N$.

    Let $\{\bfT_i\}_{i=1}^m$ be a collection of maximal $k$-tori that generate $\bfL$. 
    We get that the subgroup $A=\cap_{i=1}^m(V/N)^{T_i}$ is $\w'$-open in $V/N$. 
    Let $\pi\colon V \to V/N$ denote the projection map.
    Note that by construction
$$\pi^{-1}(A)=\{v\in V \ : \forall l\in L  \ [l,v]\in N \}$$
    We claim that $\pi^{-1}(A)$ does not contain a non-trivial $\bbQ_p$-vector subspace of $V$.
    Assume by contradiction that $v\in \pi^{-1}(A)$ is a non trivial element contained in a vector subspace of $\pi^{-1}(A)$.
    Since $Z(G)$ is compact,
    there exists $l\in L$ such that $[l,v]\ne e$.
    By $\bbQ_p$-linearity, $\lambda [l,v]=[l,\lambda v]\in N$, so by taking $\lambda\in \bbQ_p$ with a large absolute value we get a contradiction to the the fact that $N$ is bounded.
    Since $\pi^{-1}(A)$ does not contain a $\bbQ_p$-vector subspace, by Proposition~\ref{prop : not containing subspace compact*} it is compact.
    Thus $A\subset V/N$ is finite.

    Finally, we got that the set $A$ is a finite $\w'$-open subset in $V/N$, hence the $\w'$-topology is equal to the analytic quotient topology which is discrete.
\end{proof}
\begin{lemma}
\label{lemma : comutative radical prime field sealed}
    Let $k=\bbQ_p$ or $k=\bbR$. Let $\bfG$ be a connected $k$-algebraic group with Levi decomposition $\bfL\ltimes \Rad{u}{\bfG}$ over $k$, such that $\Rad{u}{\bfG}$ is commutative.
    Assume that $Z(L)$ compact, and that the action of $L$ on $\Rad{u}{G}$ is with no invariant vectors. Then $G$ is sealed.
\end{lemma}
\begin{proof}
The abelian group $\Rad{u}{G}$ is isomorphic to a finite-dimensional $k$-vector space, which will be denoted $V=\Rad{u}{G}$.
    By complete reducibility of reductive groups, the vector space $V$ decomposes into finitely many irreducible  sub-representations $V=\oplus_{i=1}^m V_i$.
    We prove by induction on the length of the representation $m$, that $G$ is sealed.
    
    Assume $m=1$. $G=L\ltimes V$, where
    the action of $L$ on $V$ is irreducible and $k$-algebraic. Thus $C=\{ x\in L : \ \forall v\in V \ [x,v]=e\}$, the kernel of the action, is a normal $k$-subgroup of $L$. Thus $C$ is $k$-points of a normal $k$-algebraic subgroup $\mathbf{C}\lhd \bfL$. Since $\bfL$ is reductive with a compact center, by Corollary~\ref{corollary : normal subgruop of reductive sealed is sealed }, the normal subgroup $C$ is sealed.

    Let $H$ be a topological group and $f\colon G\to H$ a continuous homomorphism with a dense image. The subgroup $f(C)$ is a closed normal subgroup of $H$, making $H/f(C)$ a topological group.

    Let $f_c\colon G/C \to H/f(C)$ be the factor map induced by $f$. 
    Since $G/C=L/C\ltimes V$ and $L/C$ acts on $V$ faithfully irreducible, by Proposition~\ref{prop : faithfull irreducable action normal subgroup contined} either $\ker(f_c)\subset V$ or $V\subset \ker(f_c)$. In case $V\subset \ker(f_c)$, the image of $f$ is equal to the image of $L$ and since $L$ is sealed, $f(G)=f(L)$ is closed.
    Thus we assume as we may that $\ker(f_c)\le V$.
    Since the action is irreducible, the $L$-invariant subset $\ker(f_c)$ does not contain a vector subspace.
    If $k=\bbR$ then $\ker(f_c)$ must be discrete and since it is $L$ invariant it must be trivial.
    In case $k=\bbQ_p$,  
    by Proposition~\ref{prop : not containing subspace compact*},
    $\ker(f_c)$ is compact. Moreover, by irreducibility the subset $\ker(f_c)$ contains a basis,  therefore $\ker(f_c)$ is open.
    Following Lemma~\ref{lem : selaed open subgroup unipotent radical} we get that $V/\ker(f_c)$ is minimal in $L\ltimes V/\ker(f_c)$, hence by Corollary~\ref{cor : minimality under qutient}, $V/\ker(f_c)$ is minimal in $G/C$.
    Observe the following commutative diagram
    \[
    \begin{tikzcd}
        G \rar["f"]\dar[] & H \dar[]\\
        L/C\ltimes V \dar[]\rar["f_c"] & H/f(C) \\
        L/C\ltimes V/\ker(f_c)\urar[hookrightarrow, dashed] 
    \end{tikzcd}
    \] 
    By the relative minimality of $V/\ker(f_c)$,
    the preimage of $f_c(V/\ker(f_c))$ in $H$ is closed and is equal to $f(C\times V)$.
    The group $C\times V$ is normal in $G$ hence $f(C\times V)$ is a closed normal subgroup of $H$, making $H/f(C\times V)$ a topological group.

    Let $f_{cv}\colon G/(C\times V) \to H/f(C\times V)$ be the corresponding factor map induced by $f$.
    Following proposition~\ref{prop : reductive sealed}, since $L$ has a compact center, $L$ is sealed. Observe that $L/C\cong G/(C\times V)$. 
    Consider the following commutative diagram
    \[
    \begin{tikzcd}
        G \rar["f"]\dar[] & H \dar[]\\
        L/C\rar["f_{cv}"] & H/f(C\ltimes V)
    \end{tikzcd}
    \] 
    We get that $f_{cv}(G/(C\times V))$ is a closed subgroup in $H/f(C\times V)$ and therefore its preimage in $H$, which is equal to $f(G)$, is closed. This finishes the case $m=1$.
 
    Assume $m>1$, and let $V_1\le V$ be an irreducible sub representation.
    By the induction hypothesis, $L\times V_1$ and $L\times V/V_1$ are sealed.
    Let $H$ be a topological group and $f\colon G\to H$ a continuous homomorphism with a dense image. The subgroup $f(L\ltimes V_1)$ is a closed normal subgroup of $H$, making $H/f(L\ltimes V_1)$ a topological group.
    Let $f_1\colon G/V_1 \to H/f(L\ltimes V_1)$ be the corresponding factor map induced by $f$. 
    We get that $f_{1}(G/V_1)$ is a closed subgroup in $H/f(L\times V_1)$ and therefore its preimage in $H$, which is equal to $f(G)$, is closed.
    This finishes the proof.
\end{proof}
\begin{prop}

\label{prop : sealed for commutaitve rad gemeral}
    Let $\bfG$ be a connected $k$-algebraic group with Levi decomposition $\bfL\ltimes \Rad{u}{\bfG}$ over $k$, such that $\Rad{u}{\bfG}$ is commutative.
    Assume that $Z(L)$ compact, and the action of $L$ on $\Rad{u}{G}$ is with no invariant vectors, then $G$ is sealed.
\end{prop}
\begin{proof}
    As $k$ is a local field of characteristic zero, $k$ is a separable finite field extension of $k_0=\bbQ_p$ or $k_0=\bbR$.
    Let $\mathrm{R}_{k/k_0}\bfG$ be the restriction of scalars (from $k$ to $k_0$) of the group $\bfG$.
    By functorial properties of the restriction of scalars functor, \cite[Chapter I (1.7)]{margulis1991discrete} we get that $\mathrm{R}_{k/k_0}\bfL \ltimes \mathrm{R}_{k/k_0}\Rad{u}{G} $
    is a $k_0$-Levi decomposition of $\mathrm{R}_{k/k_0}\bfG$.
    Moreover $\mathrm{R}_{k/k_0}\bfG(k_0)\cong \bfG(k)$. The center of $\mathrm{R}_{k/k_0}\bfL(k_0)=L$ is compact and the action of $L$ on $\mathrm{R}_{k/k_0}\Rad{u}{\bfG}(k_0)=\Rad{u}{G}$ is with no invariant vectors, hence by Lemma~\ref{lemma : comutative radical prime field sealed} $G$, is sealed.
\end{proof}
\begin{proof}[Proof of Theorem~\ref{theorem :  sealed theorem}]

Denote by $\{\bfZ_i\}_{i=1}^m$ the upper central series of $\Rad{u}{\bfG}$.
By Zariski density of the $k$-points, the collection $\{Z_i\}_{i=1}^m$ is the upper central series of $\Rad{u}{G}$.

Implication (1) to (2) is straightforward. Indeed, assume that $G$ is sealed since $G$ is Polish and any quotient group is sealed and in particular minimal. In particular, for any unipotent normal $k$-algebraic subgroup $\bfN\le \bfG$ the group $G/N=\bfG/\bfN(k)$ (see Proposition~\ref{prop : qutient split solvable}) is minimal and by Lemma~\ref{lem : minimal has compact center}, its center is compact.

Assume (2), observe that $L=G/\Rad{u}{G}$ hence its center is compact. We will show that the action of $L$ on $\Rad{u}{G}$ is with no non-trivial fixed points. Assume by contradiction that there exists a non-trivial element $u\in \Rad{u}{G}$ such that for any $l\in L$, $[l,u]=e$. Let $i$ be the minimal index such that $u\in Z_i\setminus Z_{i-1}$. Observe that $\bfZ_{i-1}$ is a normal unipotent $k$-algebraic group hence by assumption $Z(G/Z_{i-1})$ is compact. But then $u Z_{i-1} \in Z(G/Z_{i-1}) $ which is a contradiction since compact algebraic groups have no unipotent elements. Hence the action of $L$ on $\Rad{u}{G}$ is with no non-trivial fixed points and (3) holds.

By the Lie-correspondence for unipotent algebraic groups in characteristic zero \cite[Chapter I, Proposition (1.3.1)]{margulis1991discrete}, the logarithmic mapping $\Rad{u}{\bfG}\to \Lie(\Rad{u}{\bfG}$ is an $\bfL$ equivariant bijection.
Hence a non-trivial fixed point in the unipotent radical will be mapped to a non-trivial invariant vector in the Lie-Algebra, and vise versa, therefore (3) and (4) are equivalent.

    We will prove that condition (3) implies (1) by induction on $n$, the nilpotency degree of the unipotent radical $\Rad{u}{\bfG}$.
    If $n=0$, that is $\Rad{u}{\bfG}=\{e\}$, then the group $\bfG$ is reductive and $G$ has compact center, it follows by Proposition~\ref{prop : reductive sealed} that $G$ is sealed.
    If $n=1$, that is $\Rad{u}{\bfG}$ is abelian, it follows by Proposition~\ref{prop : sealed for commutaitve rad gemeral} that $G$ is sealed.
    Assume that the nilpotency degree of $\Rad{u}{\bfG}$ is $n>1$. 
    Set $\bfG_1 = \bfL \ltimes \bfZ_1$, and $\bfG^1 = \bfL \ltimes \Rad{u}{\bfG} / \bfZ_1$. We will show that condition (3) holds for the $k$-groups $\bfG_1$ and $\bfG^1$.
    Clearly $Z(L)$ is compact, and as $\bfZ_1\le \Rad{u}{\bfG}$, the action of $\bfL$ on $\bfZ_1$ is with non-trivial fixed point.
    We are left to show that the $\bfL$-action on $\Rad{u}{\bfG} / \bfZ_1$ is with no non-trivial fixed points.
    
    Assume by contradiction that there exists a non-trivial fixed point in $\Rad{u}{\bfG} / \bfZ_1$, i.e. there exists an element $u\notin\bfZ_1$ such that for any $l\in \bfL$, $[l,u]\in \bfZ_1$.
    The commutation map 
    \[
    \mathbf{c}_u \colon \bfL \to \bfZ_1 \ \  \mathbf{c}_u(l)=[l,u]
    \]
    satisfies the following co-cycle relation
    \[
       c_u(xy)=xyuy^{-1}x^{-1}u^{-1} = x(yuy^{-1}u^{-1})x^{-1} (x ux^{-1}u^{-1})
         = c_u(x)+x.c_u(y)
    \]
    with respect to the conjugation action of  $\bfL$ on $\bfZ_1$.
    By \cite[Chapter VII, Lemma (5.22)]{margulis1991discrete} the first co-homology of $\bfL$ is trivial, therefore there exist $z\in \bfZ_1$ such that $ \mathbf{c}_u (l)=l.z-z=[l,z]$. We get that for any $l\in \bfL$ 
    \begin{align*}
        e &= [l^{-1},u][l^{-1},z]^{-1} \\
         &= [l^{-1},u][z,l^{-1}] \\
        &=  l^{-1}z^{-1} uz l u^{-1}z l^{-1} z^{-1}l\\
        &=  l^{-1}z^{-1} [uz, l] z^{-1} l\\
     &=  l^{-1} [uz, l]  l\\
    \end{align*}
    hence for any $l\in \bfL$, $[uz,l]=e$ in contradiction to $\Rad{u}{\bfG}$ having no non-trivial fixed point.
    This shows that the $\bfL$-action on $\Rad{u}{\bfG} / \bfZ_1$ is indeed with no non-trivial fixed points.
   
    Since the nilpotency degrees of $\Rad{u}{\bfG_1}$ and of $\Rad{u}{\bfG^1}$ are strictly smaller then $n$, we get by the induction hypothesis that $G_1$ and $G^1$ are sealed.
  
    Finally, let $f\colon G\to H$, be a continuous homomorphism into a topological group $H$ with a dense image.
    The normal subgroup $G_1\lhd G$ is sealed, hence $f(G_1)$ is a normal closed subgroup of $H$, making $H/f(G_1)$ a topological group.
    Since $Z_1\le G_1$ the map $f$ factors through  $f_{1}\colon G/(Z_1) \to H/f(G_1)$.
    Observe that $G/Z_1\cong G^1$. Consider the following commutative diagram
    \[
    \begin{tikzcd}
        G \rar["f"]\dar[] & H \dar[]\\
        G^1\rar["f_{1}"] & H/f(G_1)
    \end{tikzcd}
    \] 
    The subgroup $f_{1}(G^1)$ is a closed subgroup in $H/f(G_1)$ and therefore its preimage in $H$, which is equal to $f(G)$, is closed. This means that $G$ is sealed, and the proof is complete.
\end{proof}
\section{A note on positive characteristic}
\label{Positive char}
As a proof of concept, we will show how the tools developed in the previous sections can be used to show the minimality of semisimple groups over local fields of positive characteristics. 

Similar to the proof of Proposition~\ref{prop : minimality for reductive}, we show that minimal-parabolic subgroups are minimal as topological groups and conclude minimality by co-compactness of parabolic subgroups.
Throughout this section we let $k$ be a local field of characteristic $p$. 

\begin{lemma}
\label{lemma : rel minimality of miniaml parabolic in reductive grp}
    Let $\bfG$ be a connected reductive $k$-group, and let $P\le G$ be a minimal parabolic in $G$. Then the unipotent radical $\Rad{u}{\bfP}$ is defined over $k$ and the group of $k$-points $\Rad{u}{P}$ is minimal in $P$.
\end{lemma}
\begin{proof}
    By the structure of parabolic subgroups \cite[Proposition 20.6]{borel2012linear}, the minimal parabolic subgroup $\bfP\le \bfG$ is equal to the semi-direct product $\CC_\bfG(\bfT)\ltimes \Rad{u}{\bfP}$, where $\bfT$ is a maximal $k$-split torus, and $\CC_\bfG(\bfT)$ is the centralizer of $\bfT$ in $\bfG$.
      
    By \cite[Chapter I, Proposition (1.3.2)]{margulis1991discrete}, the unipotent group $\Rad{u}{\bfP}$ is defined and split over $k$ and there exists a $\bfT$-equivariant $k$-isomorphism $\Phi \colon \Rad{u}{\bfG}\to \Lie(\Rad{u}{\bfG}) $.
    Hence pulling the $k$-linear structure from the Lie-algebra action to the center of the unipotent radical, we get that $Z=Z(\Rad{u}{P})$ is isomorphic to a finite-dimensional $k$-vector space admitting a $T$ linear action.
    Since the torus is $k$-split, the linear space $Z$ admits a weight space decomposition with respect to non-trivial characters $\chi_i\colon T \to k^*$.
    This means $Z=\bigoplus_{i}V_i$ such that for any $v_i\in V_i$ and $t\in T$, $t.v_i=\chi_i(t)\cdot v_i$.
    Following \cite[Corollary 1. 8.11]{borel2012linear}, for any $i$, there exists a co-character $\lambda_i\colon k^*\to T$  such that the composition $\psi_{m_i}= \chi_i\circ\lambda_i$ is the power map
    $$
    \psi_{m_i}\colon  k^*\to  k^*, \ \ \psi_{m_i}\colon  x\mapsto x^{m_i} 
    $$
    Writing $m_i=p^{N_i}\cdot d_i$ with $N_i$ equal to the $p$-adic valuation of $m_i$, we define the local field $k_i=k^{p^{N_i}}$. By \cite[Part II, Chapter II, Inverse Function Theorem ]{serre2009lie} the map
    $$
    \psi_{d_i}\colon k^*\to k^*,\ \  \psi_{d_i}\colon x\mapsto x^{d_i}
    $$ 
    is a local homeomorphism at $1_k$. Moreover, the Fourbenuius map 
    $$
    \psi_{p^{N_i}}\colon k^*\to k_i^*,\ \  \psi_{p^{N_i}}\colon x\mapsto x^{p^{N_i}}
    $$ 
    is a homeomorphism from $k^*$ onto $k_i^*$.

    We will show relative minimality of $\Rad{u}{P}$ in $A=T\ltimes \Rad{u}{P}\le P$ and thus obtain relative minimality of $ \Rad{u}{P}$ in $P$.
    Let $\w$ be a group topology on $A$ weaker than the analytic topology.
    We claim that for any $i$, the subgroup $V_i$ is minimal in $A$.
    
    Following Remark~\ref{remark: stronger action continuous}, furnished with the analytic topology, the group $T$ acts continuously on $V_i$ endowed with the $\w$-subspace topology.
    We get that the action of $k^*$ on $V_i$ defined by $x.v=x^{m_i}v$ is continuous, since it is the composition of the co-character $\lambda_i$ with the $T$ action via the character $\chi_i$.
    Since $k$ is a finite filed extension of $k_i$, $V_i$ is a finite-dimensional $k_i$-vector space. Denote by $m\colon k_i\times V_i \to V_i$ the scalar multiplication. By the following commutative diagram
    \[
    \begin{tikzcd}
        k^*\times V_i \drar["\pi"] \dar["\psi_{m_i}",swap]&  \\
        k_i^*\times V_i \rar["m"] & V_i
    \end{tikzcd}
    \]
    there exists an open subgroup $1\in W\le k^*$ such that $\psi_{m_i}(W)\le {k_i}^*$ is an analytic subgroup acting continuously by multiplication on $V_i$.
    Thus by Lemma~\ref{lem : main result scalar multiplication}, the topology on $V_i$ is equal to the analytic $k_i$-point topology which is equivalent to the analytic $k$-point topology.
    
    We got that for any $i$ the subgroup $V_i$ is minimal in $A$. 
    Following Corollary~\ref{cor : product of relitive minimal}, the subgroup $Z=\oplus_i V_i$ is minimal in $A$.
    Therefore, the $\w$-subspace topology on $Z$ is equal to the analytic topology. Thus, by Lemma~\ref{lem : nilpotent center minimal} the $\w$-subspace topology on $\Rad{u}{P}$ is equal to the analytic topology.
\end{proof}
\begin{lemma}
\label{lemma : minimality of miniaml parabolic in reductive grp (Psitive char)}
     Let $\bfG$ be a semisimple $k$ group, and let $P\le G$ be a minimal parabolic in $G$. Then for any central subgroup $Z_0$, the group $P/Z_0$ is minimal.
\end{lemma}
\begin{proof}
    Let $A=T\ltimes \Rad{u}{P}$ be as in the proof of Lemma~\ref{lemma : rel minimality of miniaml parabolic in reductive grp} above were we established relative minimality of the unipotent radical $\Rad{u}{P}$.
    Using the same methods as in Lemma~\ref{lem : lemma for minimality of solvable}, the $\bfT$-equivariant $k$-morphism from $\Rad{u}{\bfP}$ to its Lie algebra $ \Lie(\Rad{u}{\bfP})$ fosters a $k$-rational representation $\pi\colon A \to \GL(V)$ with kernel equal to $C\ltimes \Rad{u}{P}$, where $C$ is the kernel of the $T$ action.
    The proof of the minimality of $P/Z_0$, now follows the same as in the
    proof of Lemma~\ref{lemma : minimality of miniaml parabolic in reductive grp}.
\end{proof}
We, therefore, obtain the general result for semisimple groups over local fields.
\begin{thm}[Bader-Gelander, Theorem 5.1 \cite{MR3692904}]
    Let $k$ be a local field, and $\bfG$ a semisimple $k$-group. Then $G=\bfG(k)$ is sealed.
\end{thm}
\begin{proof}
   The theory developed by Borel and Tits on the structure of $G^+$ is valid in positive characteristic as well. Therefore, the proof follows exactly as in \S \ref{subsecting sealed reductive}, where we first obtain the sealed property for almost $k$-simple groups by minimality of $P/Z_0$ where $P$ is a minimal parabolic and $Z_0$ is a central subgroup, see Lemma~\ref{lem : simple groups are sealed}.
   Products and quotients of sealed groups are sealed. As semisimple $k$-groups are a quotient of a product of $k$-almost simple groups (see \cite[Theorem 22.10]{borel2012linear}), we get that the $k$-points of semisimple $k$-group is sealed.
\end{proof}
    
\bibliographystyle{alpha}
\bibliography{endtobig.bib}
\end{document}